\documentclass[12pt]{article}
\usepackage{graphicx}
\usepackage{epsfig}
\usepackage{amsfonts}
\usepackage{amssymb}
\usepackage{amsthm}
\usepackage{amsmath}		
\usepackage{tikz}
\usepackage{pdflscape}
\usetikzlibrary{decorations.pathreplacing}
\usetikzlibrary{patterns}
\newtheorem{thm}{Theorem}[section]
\newtheorem{cor}[thm]{Corollary}
\newtheorem{lem}[thm]{Lemma}
\newtheorem{prop}[thm]{Proposition}
\theoremstyle{remark}
\newtheorem{rem}[thm]{Remark}
\theoremstyle{definition}
\newtheorem{defi}[thm]{Definition}
\newtheorem{notation}[thm]{Notation}
\newcommand{\ins}{\mathbf{Insert}}
\newcommand{\patt}{\mathbf{Patt}}
\newcommand{\dom}{\mathbf{Dom}}
\newcommand{\pprec}{\prec\!\prec}

\newlength{\hatchspread}
\newlength{\hatchthickness}
\tikzset{hatchspread/.code={\setlength{\hatchspread}{#1}},
         hatchthickness/.code={\setlength{\hatchthickness}{#1}}}
\tikzset{hatchspread=5pt,
         hatchthickness=0.2pt}
\pgfdeclarepatternformonly[\hatchspread,\hatchthickness]
   {custom north west lines}
   {\pgfqpoint{-2\hatchthickness}{-2\hatchthickness}}
   {\pgfqpoint{\dimexpr\hatchspread+2\hatchthickness}{\dimexpr\hatchspread+2\hatchthickness}}
   {\pgfqpoint{\hatchspread}{\hatchspread}}
   {
    \pgfsetlinewidth{\hatchthickness}
    \pgfpathmoveto{\pgfqpoint{0pt}{\hatchspread}}
    \pgfpathlineto{\pgfqpoint{\dimexpr\hatchspread+0.15pt}{-0.15pt}}
    \pgfusepath{stroke}
   }

\begin{document}

\title{Structure of Random 312-Avoiding Permutations}
\author{Neal Madras \\ Department of Mathematics and Statistics \\
York University \\ 4700 Keele Street  \\ Toronto, Ontario  M3J 1P3 Canada 
\\  {\tt  madras@mathstat.yorku.ca}  \\ and \\ Lerna Pehlivan  \\ Department of Mathematics \\
University of Washington \\Seattle, WA 98195-4350 USA \\  {\tt  pehlivan@math.washington.edu} }
\maketitle

\begin{abstract} 
We evaluate the probabilities of various events under the uniform distribution on the set of 
$312$-avoiding permutations of $1,\ldots,N$.  
We derive exact formulas for the probability that the $i^{th}$ element of a random permutation is a 
specific value less than $i$, and for joint probabilities of two such events.
In addition, we obtain asymptotic approximations to these probabilities for large $N$ when the 
elements are not close to the boundaries or to each other. 
We also evaluate the probability that the graph of a random $312$-avoiding permutation 
has $k$ specified 
decreasing points, and we show that for large $N$ the points below the diagonal look 
like trajectories of a random walk.  
\end{abstract}

\section{Introduction}
Let $S_N$ denote the set of permutations of numbers $1,\ldots, N$ for each positive integer $N$. Given $\tau \in S_k$ (with $k\leq N$), we say that a permutation 
$\sigma=\sigma_1\ldots\sigma_N$ avoids the pattern $\tau$ (or ``$\sigma$ is $\tau$-avoiding'') if
there is no subsequence of $\sigma$ with length $k$ having the same relative order as $\tau$. The set of $\tau$-avoiding permutations in $S_N$ is denoted by $S_N(\tau)$. For example the permutation $435621$ avoids the $312$ pattern and hence $435621 \in S_6(312)$ but it is not an element of $S_6(321)$ since 432, 431, 421, 321, 521, and 621 are subsequences in $435621$ 
having the $321$ pattern.
A permutation $\sigma=\sigma_1\ldots\sigma_N$ can be represented as a function $\sigma $ that maps $i$ to $\sigma(i)=\sigma_i$. The graph of this function is the set of $N$ points $\left\{(i,\sigma_i):i=1\ldots N\right\}$ (see Figure \ref{fig:exp}). 
Points of the form $(i,\sigma_i)=(i,i)$ are said to be on the diagonal of the graph of $\sigma$ (these correspond to fixed points of the permutation).

\begin{center}
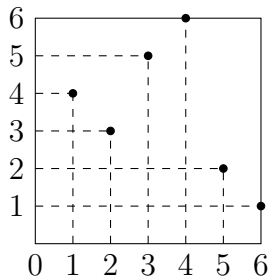
\begin{figure}[htp]
\centering
  \begin{tikzpicture}[scale=0.5]
    \draw (0,0) rectangle (6,6);
    \draw [fill] (1,4) circle (0.1);
    \draw [fill] (2,3) circle (0.1);
    \draw [fill] (3,5) circle (0.1);
    \draw [fill] (4,6) circle (0.1);
    \draw [fill] (5,2) circle (0.1);
    \draw [fill] (6,1) circle (0.1);

    \draw [dashed] (1,0) -- (1,4);
    \draw [dashed] (2,0) -- (2,3);
    \draw [dashed] (3,0) -- (3,5);
    \draw [dashed] (4,0) -- (4,6);
    \draw [dashed] (5,0) -- (5,2);
    
    \draw [dashed] (0,4) -- (1,4);
    \draw [dashed] (0,3) -- (2,3);
    \draw [dashed] (0,5) -- (3,5);
    \draw [dashed] (0,2) -- (5,2);
    \draw [dashed] (0,1) -- (6,1);

    \node [below] at (0,0) {$0$};
    \node [below] at (1,0) {$1$};
    \node [below] at (2,0) {$2$};
    \node [below] at (3,0) {$3$};
    \node [below] at (4,0) {$4$};
    \node [below] at (5,0) {$5$};
    \node [below] at (6,0) {$6$};
    \node [left] at (0,1) {$1$};
    \node [left] at (0,2) {$2$};
    \node [left] at (0,3) {$3$};
    \node [left] at (0,4) {$4$};
    \node [left] at (0,5) {$5$};
    \node [left] at (0,6) {$6$};
  \end{tikzpicture}
  \caption{$\sigma=435621 \in S_6$ viewed as a function $i \mapsto \sigma_i$. This is a 
  $312$-avoiding permutation.}
	\label{fig:exp}
  \end{figure}
  \end{center}

The study of pattern-avoiding permutations often reveals connections to other combinatorial objects.  
For example in \cite{bona2} it is shown that there is a bijection between $1342$-avoiding permutations and plane forests of $\beta(0,1) $-trees,  as well as  with ordered collections of rooted bicubic planar maps. 

Among other well studied connections of permutations excluding or including certain patterns in the literature are Kazhdan-Lusztig polynomials, singularities of Schubert varieties, Chebyshev polynomials, rook polynomials for Ferrers boards.  In a recent book \cite{kitaev} Kitaev goes through a vast amount of literature to point out the connections of permutations with other mathematical objects. Also, Bouvel and Rossin \cite{bouvel} show how permutation patterns are related to problems in computational biology.     

One of the initial motivations to study pattern avoiding permutations came from computer science. 
A basic problem in computer science is sorting $n$ distinct elements in increasing order. Stack sorting is an algorithm that does the sorting operation efficiently, although it only works
on some permutations. It was observed that a permutation is stack sortable if and only if it
avoids the pattern $231$. A detailed explanation of the connection can be found in B{\'o}na's book \cite{bona1}. An important reference on stack sorting is Knuth's book \cite{knuth}. 

 Two recent papers of Madras and Liu \cite{madras2} and Atapour and Madras \cite{madras}  
 present numerical and probabilistic approaches to  investigate the shapes of random pattern avoiding permutations, mainly of length three, four and five. Both papers include Monte Carlo simulations 
 suggesting the limiting distributions of the positions of points of the permutations. For example  
Atapour and Madras \cite{madras} present the result of Monte Carlo simulation 
(similar to Figure \ref{fig:MonteCarlo} here) 
which suggests that typical $312$-avoiding permutations are accumulated near the
diagonal as well as below the diagonal. 
To generate random $312$-avoiding permutations, they run 
a Markov chain on $S_N(312)$ which is irreducible, symmetric
and aperiodic \cite{madras2}, and hence has the uniform distribution on $S_N(312)$ 
as its limiting distribution.
They further prove the following results that support the findings of the simulations. 

	\begin{figure}[htbp]
		\centering
		\includegraphics[width=.7\textwidth,keepaspectratio ]{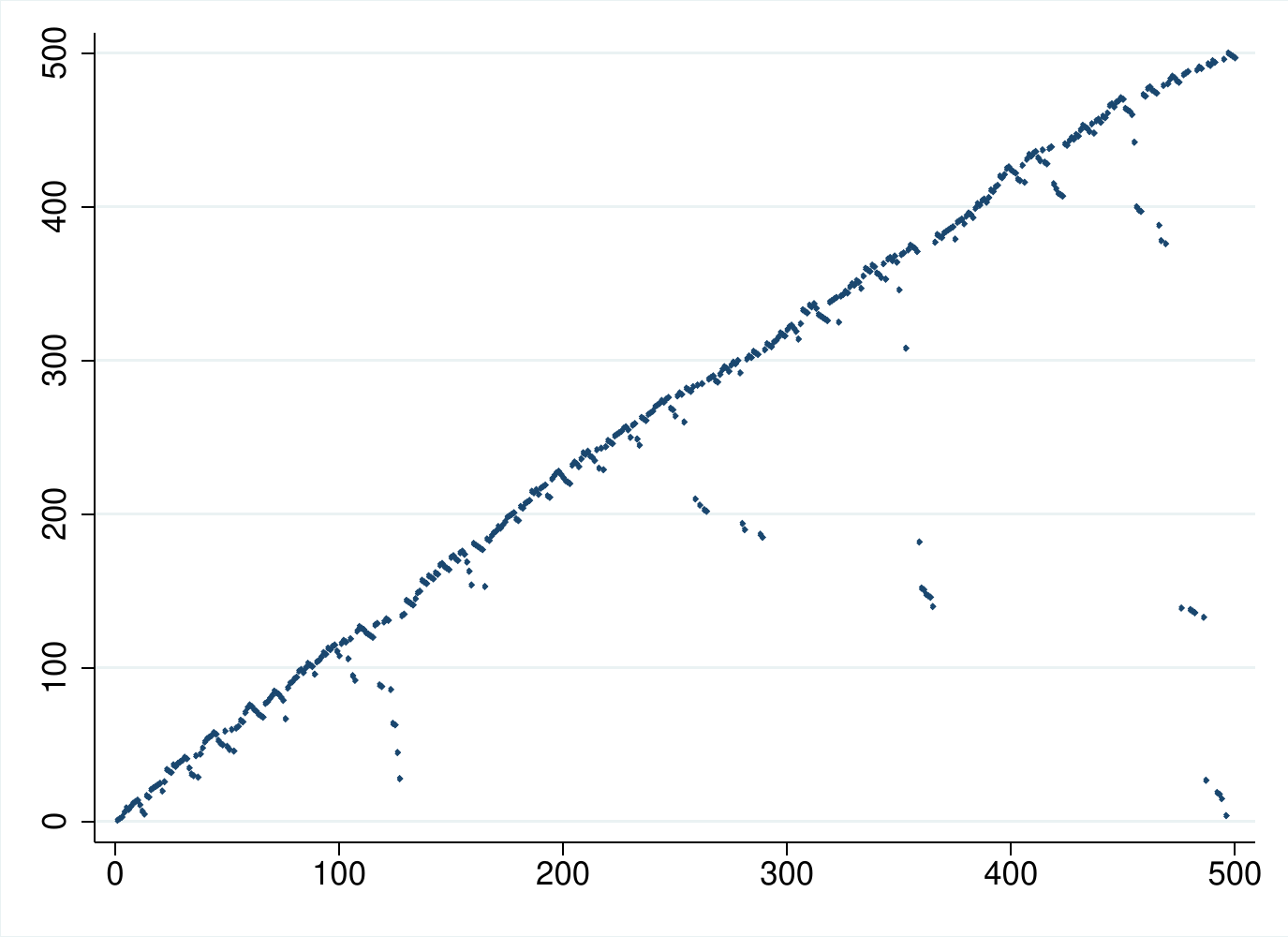}
		\caption{Graph of a randomly generated 312-avoiding permutation with N=500} 
		\label{fig:MonteCarlo}
	\end{figure}

\begin{defi}\label{uniform}
Consider a pattern $\tau$. For each $N \geq 1$, let $P_N^{\tau}$ be the uniform distribution on the set $S_N(\tau)$; that is $P_N^{\tau}(A)\,=\,\left|A\right|/\left|S_N(\tau)\right|$. For simplicity  
we shall write $P_N$ to denote $P_N^{312}$ throughout this paper.
\end{defi}   
\begin{thm}\label{madras}\cite{madras}
Define the function $K^*$ on the unit square $[0,1]^2$ by
$$K^{*}(s,t)=\left\{ \begin{array}{cll}
		1                  &\mbox{ if } 0 \leq t \leq s \leq 1,\\ 
		\frac{1}{4}\frac{(2-s-t)^{2-s-t}(s+t)^{s+t}}{(1-s)^{1-s}(1-t)^{1-t}t^ts^s} &\mbox{ if } 0 \leq s \leq t \leq 1. 
		\end{array}\right.$$
Then for any  relatively open subset $D$  of $[0,1]^2$, we have
   \begin{equation*}
      \lim_{N \rightarrow \infty} \left[P_N\left\{\left(\frac{i}{N},\frac{\sigma_i}{N}\right) \in D \mbox{ for some } i \in \left\{1,\ldots,N \right\}  \right\}\right]^{1/N} = \;
       \sup \left\{K^{*}(s,t): (s,t)\in D  \right\}.
      \end{equation*}
\end{thm}
Theorem \ref{madras} concludes that it is rare to have points of the graph well above the diagonal (since $K^*(s,t)<1$ for $0<s<t<1$), but  it is not rare to have points well below the diagonal.
A related result is the following, which says 
that the number of points well below diagonal is  $o(N)$ with high probability.

\begin{prop}\cite{madras}
Let $\delta > 0$, $0 < t < 1$ and $K_N(\sigma,\delta N)= \left| \left\{i: \sigma_i < i - \delta N \right\} \right|$.  Then
           $$ \lim_{N \rightarrow \infty} [P_N(K_N(\sigma, \delta N) > tN)]^{1/N} < 1 .$$
\end{prop}

The present paper is mainly motivated by these results in \cite{madras}. 
In this paper we investigate the probabilities of having a $312$-avoiding permutation that has 
one or two specified points below the diagonal (i.e., satisfying $\sigma_i=j$ for specified $i$ and 
$j$ with $j<i$). 
We also extend our results to $k$ decreasing points below the diagonal.  
Exact evaluations of the probabilities 
and the approximation results for these probabilities for large $N$ are stated in the next section. 
Our main theorems imply that the probability of obtaining $312$-avoiding permutations with one
specified point below the diagonal is of order $N^{-3/2}$, and the probability of obtaining 
$312$-avoiding permutations with two (well separated) specified points below the diagonal 
is of order $N^{-3}$.  
However, the two-point probability is not approximated by the product of the corresponding 
one-point probabilities.  
In particular, Corollary \ref{cor2}
describes situations in which two one-point events are positively or negatively correlated.
Exact combinatorial results for 312-avoiding permutations with specified points above the diagonal
could also be calculated in similar manner; however in this paper we concentrate our attention 
below the diagonal, as motivated by \cite{madras}.  

 While this paper was being written, S. Miner and I. Pak completed a preprint (now \cite{pak}) 
that investigates probabilities of random $123$- and $132$-avoiding permutations with one 
specified point and their asymptotics.  In particular, \cite{pak} independently proves Theorems \ref{propS_A(N,N-t,j)}
and \ref{propapprox1}, and extends these results considerably in directions that we have 
not pursued.
 
 This paper is organized as follows. Section 2 collects the main results of this paper. Section 3 states the definitions, the basic terminology and results needed for the remaining parts of the paper. Section 4 consists of the proofs of Theorems \ref{propS_A(N,N-t,j)}, \ref{propS_B(N)} and \ref{propS_C(N)} which give exact formulas for probabilities of obtaining a $312$-avoiding permutation that has 
 one or
 two  specified points below the diagonal. 
 Theorem \ref{propS_A(N,N-t,j)} treats the one point case. Theorem \ref{propS_B(N)} 
 considers the case that $\sigma_{i_1}=j_1 >\ldots > \sigma_{i_k}=j_k$ with
 $j_k < \ldots < j_1 < i_1 <\ldots< i_k $. Finally, Theorem \ref{propS_C(N)} is the case that 
 $\sigma_{i_1}=j_1<\sigma_{i_2}=j_2$ with  $j_1 < i_1<i_2$ and $j_2 < i_2$.
 Section 5 gives the proofs of Theorems \ref{propapprox1}, \ref{propapprox2} and \ref{propapprox3},
 which are asymptotic approximations of the probabilities calculated in Theorems 
 \ref{propS_A(N,N-t,j)}, \ref{propS_B(N)} and  \ref{propS_C(N)} respectively. 
Section 6 proves results related to the limiting distribution of $\sigma$ near the lower right corner of the square $\left[1,N\right]^2$, as well as the limiting conditional distribution of $\sigma$ northwest 
of a given point $(i,j)$ below the diagonal given that $\sigma_{i}=j$.

\section{Main Results}   \label{sec-results}

 Let $C_N$ denote the Catalan number,  $C_N = \frac{1}{N+1}\binom{2N}{N}$ for $N \geq 0$. The proof of the well known result that  $\left| S_N(\tau)\right| = C_N$ for $\tau \in S_3$ and $N \geq 1$ can be found in ~\cite{bona1} or ~\cite{simion}. 

\begin{defi}\label{defiS_A(N,N-t,j)}  \label{def-SNij}
	For $N \in \mathbb{N}$ and $i,j\in[1,N]$, let
\[     S^{\bullet}(N,i,j)\;=\;\left\{\sigma \in S_N(312): \sigma_{i}=j \right\}    \hspace{15mm}\hbox{and}   \]
\[  S^{\Box}(N,i,j) \,=\,\{\sigma\in S_N(312):\sigma_i=j \hbox{ and }\sigma_k<j \hbox{ for all $k\in[1,i)$} \}.
\]
\end{defi}  

\begin{rem}\label{remMadras}  Observe that $|S^{\Box}(N,i,j)|=0$ if $j<i$. We also have
	\[  \left|S^{\Box}(N,i,j)\right| \;=\; \frac{(j-i+1)^2}{j(N-i+1)}\binom{2N-i-j}{N-i}\binom{i+j-2}{j-1}
	 \hspace{5mm}\mbox{whenever }1\leq i\leq j \leq N. \] 
For $i<j$, this was proven in \cite{madras}.  For the case $i=j$, we note that for every $\sigma\in S^{\Box}(N,j,j)$, 
we have $\sigma_k>j$ for every $k>j$.  Therefore
$\left|S^{\Box}(N,j,j)\right| =|S_{j-1}(312)|\times|S_{N-j}(312)|\, \,= C_{j-1}C_{N-j}\,=\,
\frac{1}{j(N-j+1)}\binom{2N-2j}{N-j}\binom{2j-2}{j-1}$. 
\end{rem}

Our first result gives the cardinality of $S^{\bullet}(N,N-t,j)$.
\begin{thm}\label{propS_A(N,N-t,j)}
	Let $j < N-t$. 
	Then
	\[ \left|S^{\bullet}(N,N-t,j)\right|= \sum_{i_0=\max \left\{1,j-t\right\}}^{j}C_{N-t-i_0}\frac{(j-i_0+1)^2}{j(t	+1)}\binom{i_0+2t-j}{t}\binom{i_0+j-2}{j-1}.
	\]
\end{thm}

\noindent
We shall prove Theorem \ref{propS_A(N,N-t,j)} by constructing a natural bijection 
between $S^{\bullet}(N,N-t,j)$ and 

\noindent
$\bigcup_{i_0=\max\left\{1,j-t\right\}}^{j} S_{N-t-i_0}(312)\times S^{\Box}(t+i_0,i_0,j)$ for fixed $N,t$ and $j$
(see Definition \ref{phiA}). 
We shall use this bijection repeatedly for the proof of Theorem \ref{propS_B(N)}, and a closely related bijection for the proof of Theorem \ref{propS_C(N)}.  

\begin{rem}\label{Bullet}
It is not hard to check that $\left|S^{\bullet}(N,N-t,j)\right| = \left|S^{\bullet}(N,N-j,t)\right|.$
\end{rem}

\begin{rem}
Theorem \ref{propS_A(N,N-t,j)} 
also gives 
the probability $P_N(S^{\bullet}(N,i,j))$, since by definition it is equal to 
$\frac{\left|S^{\bullet}(N,N-t,j)\right|}{|S_N(312)|}$. 
This applies to  Theorems \ref{propS_B(N)} and \ref{propS_C(N)} as well.
\end{rem}

\begin{notation}  In preparation for upcoming results, we state our conventions on
asymptotics.  We write $f(N)\sim g(N)$ to mean that $\lim_{N\rightarrow\infty}f(N)/g(N) =1$.
We write $f(N)\asymp g(N)$ to mean that there exists a constant $C>0$ such 
that $C > g(N)/f(N) > C^{-1}$ for all sufficiently large $N$.  We write $f(N)\pprec g(N)$
to mean that there is an $\epsilon >0$ such that $g(N)-f(N) > \epsilon N$ for all sufficiently 
large $N$.  

For example, the statement  ``$f(i,j,N)\,=\,O(N)$ for $i\pprec j$'' would mean that for
any $\epsilon>0$, there exists a $C$ and $N_1$ such that $ \left|f(i,j,N) \right|\leq CN$ for all $i$, $j$, and
$N$ such that  $j-i >\epsilon N$ and $N\geq N_1$ (where $C$ and $N_1$ can depend on 
$\epsilon$).
\end{notation}

\begin{thm}\label{propapprox1}
Fix $0< \theta < \frac{1}{6}$.  Then 
	\begin{align*}
		P_N(S^{\bullet}(N,N-t,j))
		=\frac{N^{-3/2}}{2\sqrt{\pi}\left(1-\frac{N-t-j}{N}\right)^{3/2}\left(\frac{N-t-j}{N}\right)^{3/2}}(1+O(N^{3\theta-\frac{1}{2}}))
	\end{align*}
for $0\pprec j \pprec i \pprec N$.
\end{thm}

\begin{figure}
\begin{tikzpicture}[scale=.5]
  \draw (0,0) rectangle (10,10);
  \draw [dashed] (0,0) -- (10,10);
  \draw [fill] (4,3) circle (0.125);
  \draw [dashed] (0,3) -- (4,3) -- (4,0);
  \draw [fill] (7,2) circle (0.125);
  \draw [dashed] (0,2) -- (7,2) -- (7,0);
  \node [left] at (0,3) {$j_1$};
  \node [left] at (0,2) {$j_2$};
  \node [left] at (0,10) {$N$};
  \node [below] at (0,0) {$1$};
  \node [below] at (4,0) {$N-t_1$};
  \node [below] at (7,0) {$N-t_2$};
  \node [below] at (10,0) {$N$};

  \draw (14,0) rectangle (24,10);
  \draw [dashed] (14,0) -- (24,10);
  \draw [fill] (18,3) circle (0.125);
  \draw [dashed] (14,3) -- (18,3) -- (18,0);
  \draw [fill] (21,5) circle (0.125);
  \draw [dashed] (14,5) -- (21,5) -- (21,0);
  \node [left] at (14,3) {$j_1$};
  \node [left] at (14,5) {$j_2$};
  \node [left] at (14,10) {$N$};
  \node [below] at (14,0) {$1$};
  \node [below] at (18,0) {$N-t_1$};
  \node [below] at (21,0) {$N-t_2$};
  \node [below] at (24,0) {$N$};

\end{tikzpicture}
\caption{These diagrams represents the requirement that a permutation has specified points $\sigma_{N-t_1}=j_1$ and $\sigma_{N-t_2}=j_2$ where $j_1<N-t_1$, $j_2<N-t_2$, and $N-t_1<N-t_2$. 
The left diagram corresponds to Theorem \ref{propS_B(N)} for $k=2$ and the right one corresponds to Theorem \ref{propS_C(N)}.}
\end{figure}
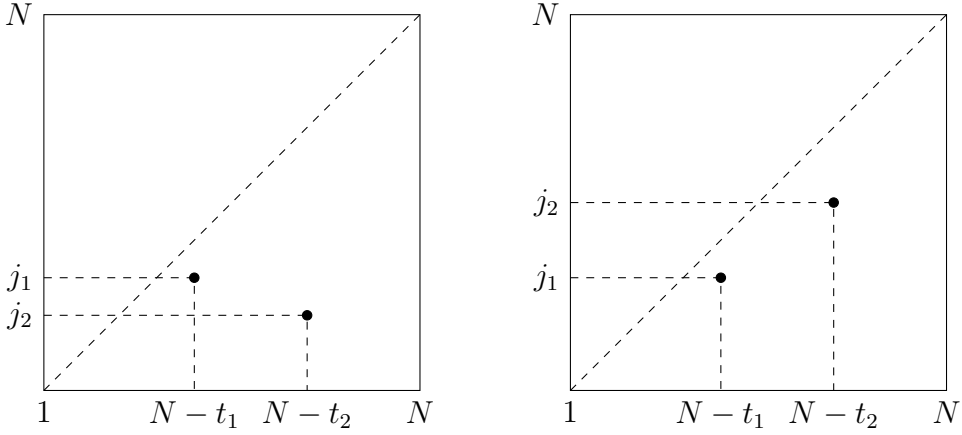

\begin{thm}\label{propS_B(N)}
 Let $j_k<j_{k-1} < \ldots < j_1 < N-t_1 < \ldots < N-t_k$ and define
\[		S^{{\searrow}_k}(N)\; \equiv\;  S^{{\searrow}_k}(N,N-t_1,\ldots,N-t_k,j_1,\ldots,j_k)=
		\left\{\sigma \in S_N(312):\sigma_{N-t_1}=j_1, \,\ldots \,\sigma_{N-t_k}=j_k \right\}.
\]
Then,
  \begin{align*}
		&\left|S^{\searrow_{k}}(N)\right|
		\;= \sum_{i=\max\left\{1,j_{k}-t_{k}\right\}}^{j_{k}}  \frac{(j_k-i+1)^2}{j_k(t_k+1)} \binom{i+2t_k-j_k}{t_k} \binom{i+j_k-2}{j_k-1} \times \\  
	& \left|S^{\searrow_{k-1}}(N-t_{k}-i,N-t_{1}-i+1\ldots,N-t_{k-1}-i+1,j_1-j_k,j_2-j_k,\ldots,j_{k-1}-j_{k})\right|.
	\end{align*}
	
	In particular for $k=2$ we have
	\begin{align*}
		\left|S^{\searrow_{2}}(N)\right| \equiv \left|S^{\searrow}(N)\right|\;=&
		\sum_{i_1=\max\left\{1,j_2-t_2\right\}}^{j_2}
		\frac{(j_2-i_1+1)^2}{j_2(t_2+1)}\binom{i_1+2t_2-j_2}{t_2}\binom{i_1+j_2-2}{j_2-1}\times \\
		&\sum_{i_0=\max\left\{1,(j_1-j_2)-(t_1-t_2-1)\right\}}^{j_1-j_2}C_{N-t_1-i_1-i_0+1}\frac{(j_1-j_2-i_0+1)^2}{(j_1-j_2)(t_1-t_2)}\times \\
		&\binom{i_0+2(t_1-t_2-1)-(j_1-j_2)}{t_1-t_2-1}\binom{i_0+j_1-j_2-2}{j_1-j_2-1}. 
	\end{align*}
\end{thm}

\begin{thm}\label{propapprox2}

Fix $0< \theta < \frac{1}{6}$.  Then 
	\begin{align*}
		&P_N\left(S^{\searrow}(N)\right)=
		\frac{1}{4\pi}\frac{N^{-3}}{\left(\frac{(N-t_2-j_2)-(N-t_1-j_1)}{N}\right)^{3/2}\left(\frac{N-t_1-j_1}{N}\right)^{3/2}\left(1-\frac{N-t_2-j_2}{N}\right)^{3/2}}(1+O(N^{3\theta-1/2}))
	\end{align*}
for $0\pprec j_2 \pprec j_1\pprec N-t_1 \pprec N-t_2 \pprec N$.
\end{thm}

\begin{thm}\label{propS_C(N)}
	For $N,j_1,j_2,t_1,t_2 \in \mathbb{N}$ such that $j_1 < N-t_1<N-t_2$, $j_1<j_2 < N-t_2$, define 
	$S^{\nearrow}(N)\equiv S^{\nearrow}(N,N-t_1,N-t_2,j_1,j_2)=\left\{ \sigma \in S_N(312):\sigma_{N-t_1}=j_1,\sigma_{N-t_2}=j_2 \right\}$.

\smallskip
\noindent
(a)  If $j_2 < N-t_1+1$, then $\left|S^{\nearrow}(N)\right|=0$.

\smallskip
\noindent
 (b) If $j_2 \geq N-t_1+1$, then
		\begin{align*}
			\left|S^{\nearrow}(N)\right|
			\;&=\sum_{i_1=\max\left\{1,j_1-j_2+N-t_1+1\right\}}^{j_1}\sum_{i_2=\max\left\{N-t_1+1,j_2-t_2\right\}}^{j_2} C_{N-t_1-i_1}C_{N-t_2-i_2}\times \\
			&\frac{(j_1-i_1+1)(j_2-i_2+1)}{j_1(t_2+1)}
			\binom{j_1+i_1-2}{j_1-1}\binom{i_2+2t_2-j_2}{t_2}\times \\
			& \left(\binom{i_1+i_2+j_2-j_1-2(N-t_1+1)}{j_2+(i_1-j_1-1)-(N-t_1)} - \binom{i_1+i_2+j_2-j_1-2(N-t_1+1)}{j_2-(N-t_1)}\right).
		\end{align*}
\end{thm}

\begin{thm}\label{propapprox3}
Fix $0< \theta < \frac{1}{6}$.	Then, 
	\begin{align*}
		&P_N(S^{\nearrow}(N))=\\
&\frac{N^{-3}}{4\pi\left(1-\frac{(N-t_1-j_1)+(N-t_2-j_2)}{N}\right)^{3/2}\left(\frac{N-t_1-j_1}{N}\right)^{3/2}\left(\frac{N-t_2-j_2}{N}\right)^{3/2}}
		(1+O(N^{3\theta-\frac{1}{2}}))
	\end{align*}
for $0\pprec j_1 \pprec N-t_1 \pprec j_2 \pprec N-t_2 \pprec N$.
\end{thm}

 The next corollary restates  Theorem \ref{propapprox1} for the special case that $j= \left\lfloor \alpha N \right\rfloor$, $i=N-t= \left\lfloor  \beta N \right\rfloor$.  Corollary  \ref{cor2} 
contains the analogous statements for Theorems
 \ref{propapprox2} and \ref{propapprox3}.
Corollary \ref{cor2} also frames these results in terms of a random field corresponding 
to points in the graph of a random $\sigma$.

\begin{cor}\label{cor1}
Assume that $0 <\alpha <\beta < 1$. Then
	\begin{align*}
		P_N(S^{\bullet}(N,\left\lfloor  \beta N \right\rfloor,\left\lfloor  \alpha N \right\rfloor))\;\sim\;
		  \frac{N^{-3/2}}{2\sqrt{\pi}(1-( \beta- \alpha))^{3/2}(\beta - \alpha)^{3/2}}.
	\end{align*}
\end{cor}

\begin{cor}\label{cor2}
Let  $i_1= \left\lfloor \beta_1 N \right\rfloor, i_2=\left\lfloor  \beta_2 N \right\rfloor, j_1= \left\lfloor \alpha_1 N \right\rfloor, j_2= \left\lfloor  \alpha_2 N \right\rfloor$. 
Also let $\Delta_1=\beta_1-\alpha_1$ and $\Delta_2=\beta_2-\alpha_2$.  
For a random  $\sigma$ having distribution $P_N$, let $Z(i,j)$ be the indicator 
of the event that $\sigma_i=j$, and let {\em Cov}$_N$ denote covariance with respect to $P_N$.
\\
(a) 
Assume that $0<\alpha_2<\alpha_1 <\beta_1<\beta_2< 1$. Then 	
\begin{align*}
		&P_N\left(S^{\searrow}(N)\right)\;\sim\;
        \frac{1}{4\pi}\frac{N^{-3}}{[(\Delta_2-\Delta_1)\Delta_1(1-\Delta_2)]^{3/2}}
        \hspace{8mm}\hbox{and}   \\
   \lim_{N\rightarrow\infty} & N^3 \,\textrm{\em Cov}_N(Z(i_1,j_1),Z(i_2,j_2))  \;=\;
    \frac{   \left(  \left[\frac{\Delta_2(1-\Delta_1)}{\Delta_2-\Delta_1} \right]^{3/2}-1\right)
    }{4\pi \,[\Delta_1\Delta_2(1-\Delta_1)(1-\Delta_2)]^{3/2}}
      \;>\; 0\,.
	\end{align*}
(b) 
Assume that $0 <\alpha_1 < \beta_1 <\alpha_2 < \beta_2 <1$. Then 
\begin{align*}
&P_N\left(S^{\nearrow}(N)\right)\;\sim\; \frac{1}{4\pi}\frac{N^{-3}}{[(1-\Delta_1-\Delta_2)
\Delta_1\Delta_2]^{3/2}}    
     \hspace{8mm}\hbox{and}   \\
   \lim_{N\rightarrow\infty} & N^3 \,\textrm{\em Cov}_N(Z(i_1,j_1),Z(i_2,j_2))  \;=\;
    \frac{   \left(  \left[\frac{(1-\Delta_1)(1-\Delta_2)}{1-\Delta_1-\Delta_2} \right]^{3/2}-1\right)
    }{4\pi \,[\Delta_1\Delta_2(1-\Delta_1)(1-\Delta_2)]^{3/2}}
      \;>\; 0\,.
\end{align*}
(c)  Assume that $0 <\alpha_1 < \alpha_2 <\beta_1 < \beta_2 <1$.  Then 
$P_N\left(S^{\nearrow}(N)\right)\,=\,0$ and 
\begin{equation}
  \lim_{N\rightarrow\infty} N^3 \,\textrm{\em Cov}_N(Z(i_1,j_1),Z(i_2,j_2))  \;=\;
    \frac{  -1
    }{4\pi \,[\Delta_1\Delta_2(1-\Delta_1)(1-\Delta_2)]^{3/2}}
      \;<\; 0\,.
\end{equation}
\end{cor}

The above asymptotic results hold for points well below the diagonal and well away from 
the sides of the square $[1,N]^2$.  

The next results concern the lower right corner of 
the square.  Our starting point is Proposition \ref{k-1-case}, of which part (\textit{a}) 
has also been observed by Miner and Pak \cite{pak}.

\begin{defi}
   \label{def-rho}
For all $a,b \in \mathbb{N}$  define 
\[
\rho(a,b) 
\; = \; \sum_{i_0=\max \left\{1,b-a+1\right\}}^{b} \frac{(b-i_0+1)^2}{ba 4^{i_0+a-1}} \binom{i_0+2(a-1)-b}{a-1} \binom{i_0+b-2}{b-1} .
\]
\end{defi}
\noindent
We note that we can also write $\rho(a,b)\,=\,  \sum_{i=\max \left\{1,b-a+1\right\}}^{b} 
| S^{\Box}(i+a-1,i,b) |/4^{i+a-1}$ (recall Definition \ref{def-SNij}
and Remark \ref{remMadras}).

\begin{prop}\label{k-1-case}
(a)  Fix $a, b \in  \mathbb{N}$. Then
\begin{align*}
\lim_{N \rightarrow \infty}
 P_N(S^{\bullet}(N,N{-}a{+}1,b) ) \;=\;  \rho(a,b).
\end{align*}   
(b) More generally, let $k\in \mathbb{N}$ and let 
$a_1,\ldots, a_{k}, b_1, \ldots b_{k} \in \mathbb{N}$.  Define  
$A_m= \sum_{l=1}^{m} a_l$, $B_m = \sum_{l=1}^{m} b_l$ for $m= 1,\ldots, k$. Then
\[
 P_N(S^{\searrow k}(N,N{-}A_{k}{+}1,\ldots,N{-}A_{1}{+}1,B_{k},\ldots,B_{1})
 ) \;= \; \left[ \prod_{v=1}^k\rho(a_v, b_{v}) \right] (1+O(N^{-1})).
\]
\end{prop}

By Remark \ref{Bullet}, we also observe that $\rho$ is symmetric in $a$ and $b$. 

Our next task is to expand part (b) above into a more complete limiting
description of $\sigma$ near the lower right corner of the square $[1,N]^2$.
Since we consider $N\rightarrow\infty$, we shall translate the lower right corner
so that the square expands to fill the second quadrant.

\begin{defi}
    \label{defi-Q}
Let $Q := \{(-i,j): i,j\in \mathbb{N}\}$ be the integer points inside the second quadrant.
For $N\in \mathbb{N}$, let $W_N :=[-N,-1]\times [1,N]$ be the $N\times N$ 
square in the lower right corner of $Q$.
For each $N$, define the collection of (dependent) binary 
random variables $\{X^N_{(-i,j)}: (-i,j)\in Q\}$ by
\[     X^N_{(-i,j)} \;=\;   \begin{cases}   1 & \mbox{if $(-i,j)\in W_N$ and $\sigma_{N-i+1}=j$ } \\
    0 & \mbox{otherwise}   \end{cases}
\]
where $\sigma$ has the distribution $P_N$.  We will often write ``$q$'' (or sometimes ``$r$''
or ``$s$'') to represent a generic
element $(-i,j)$ of $Q$; e.g.\ we can refer to the above collection as $\{X^N_q:q\in Q\}$.
\end{defi}

\setlength{\unitlength}{0.6mm}
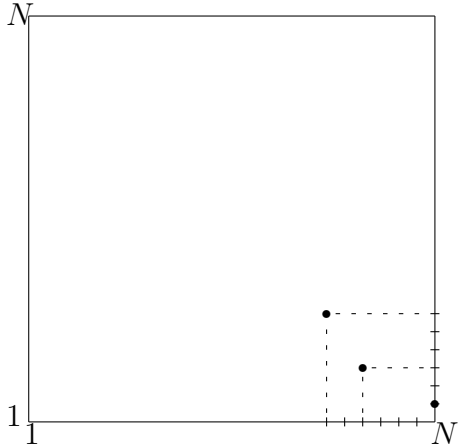
\begin{figure}
\begin{center}
\begin{picture}(100,100)
\put(5,5){\line(1,0){90}}
\put(5,5){\line(0,1){90}}
\put(5,95){\line(1,0){90}}
\put(95,5){\line(0,1){90}}
\put(95,9){\circle*{2}}
\put(79,17){\circle*{2}}
\put(71,29){\circle*{2}}
\multiput(73,29)(3.5,0){6}{\line(1,0){1}}
\multiput(71,7)(0,3.5){6}{\line(0,1){1}}
\multiput(81,17)(3.5,0){4}{\line(1,0){1}}
\multiput(79,7)(0,3.5){3}{\line(0,1){1}}
\put(91,4){\line(0,1){2}}
\put(87,4){\line(0,1){2}}
\put(83,4){\line(0,1){2}}
\put(79,4){\line(0,1){2}}
\put(75,4){\line(0,1){2}}
\put(71,4){\line(0,1){2}}
\put(94,9){\line(1,0){2}}
\put(94,13){\line(1,0){2}}
\put(94,17){\line(1,0){2}}
\put(94,21){\line(1,0){2}}
\put(94,25){\line(1,0){2}}
\put(94,29){\line(1,0){2}}
\put(4,0){1}
\put(0,4){1}
\put(94,0){$N$}
\put(0,93){$N$}
\end{picture}
\end{center}
\caption{The solid circles show three points of the graph of a permutation in 
$S^{{\searrow}3}(N,N{-}4,N{-}6,2,4,7)$.  Such a permutation gives the values
$X^N_{(-1,2)}=1=X^N_{(-5,4)}=X^N_{(-7,7)}$ (and hence $X^N_{(-7,j)}=0$ for all 
$j\neq 7$, etc.).   
\label{fig:X}
}
\end{figure}

See Figure \ref{fig:X}.
Note that the random set $\{q\in Q:  X^N_q=1\}$ is essentially the graph of 
the random permutation $\sigma$.

With the above definition, Proposition \ref{k-1-case}(b) concludes that 
the probability that $X^N_{(-A_m,B_m)}=1$ for every $m=1,\ldots,k$ converges to 
$\rho(a_1,\ldots,b_k)$ as $N\rightarrow\infty$.  Our next result builds on this to show
that there is a limiting collection of random variables indexed by points of $Q$.

\begin{thm}
   \label{thm-limX}
There exists a collection of $\{0,1\}$-valued random variables $\{X_{q}:q\in Q\}$ (with joint
distribution $P_{\infty}$) such that
for every finite subset $C$ of $Q$, the collection $\{X^N_{q}:q\in C\}$
converges in distribution to $\{X_{q}:q\in C\}$ as $N\rightarrow\infty$.
\end{thm}
 
The product form of the limit in Proposition \ref{k-1-case}(b) implies
 that the limiting collection of random variables $\{X_q:q\in Q\}$
has a kind of two-dimensional regenerative property, analogous to the more standard
regenerative property on $\mathbb{N}$ possessed by discrete renewal processes
(see \cite{Kingman}).  Our two-dimensional renewal structure is fully described 
in Theorem \ref{thm-RW}, using the following notation.

\begin{defi}
   \label{def-pi}
For $a, b \in \mathbb{N}$, let 
\begin{align*}
\pi(-a,b)\;=\;\frac{C_{a-1}C_{b-1}}{4^{a+b-1}}.
\end{align*}
\end{defi}

\noindent
It is not hard to see that $\pi$ is a probability distribution on $Q$.
(Indeed, using the well-known Catalan generating function 
$G(z) \,=\, \sum_{i=0}^{\infty}C_iz^i \,=\, (1-\sqrt{1-4z})/2z$ (e.g.\ \cite{bona1}),
we have
$\sum_{a,b=1}^{\infty}\pi(-a,b) \,=\, \frac{1}{4}\,G\left(\frac{1}{4}\right)^2\,=\,1$.)
This distribution plays a key role in the following theorem.
\begin{thm}
  \label{thm-RW}
The set $W^{*}=\{q\in Q:  X_q=1\}$ is an infinite random set of the form $\{\vec{V}_m: m\in \mathbb{N}\}$
where $\{(\vec{V}_m-\vec{V}_{m-1}):m\in \mathbb{N}\}$ are
i.i.d.\ $Q$-valued random vectors with distribution $\pi(-a,b)$ [writing $\vec{V}_0=(0,0)$].  
Moreover, the components of $\vec{V}_1$ have infinite means. 
\end{thm}
In particular, Theorem \ref{thm-RW} tells us that, with probability one, the random set 
$W^{*}=\{q\in Q: X_q=1\}$ is
an infinite sequence of points $\left\{(-A_m,B_m)\right\}$ such that the sequences
$\{A_m\}$ and $\{B_m\}$ are both strictly increasing.  Moreover, the distribution of $W^{*}$ 
is exactly that of the set of points visited by a random walk with jump distribution $\pi$.
Observe that such a random walk only jumps to the north and west.
(Here, ``random walk'' denotes a process which is the sequence of partial sums of 
an i.i.d.\ sequence of vectors.)

To illustrate our results, note that Proposition \ref{k-1-case}(b) tells us that
the probability of observing the three points shown in Figure \ref{fig:X} is 
approximately $\rho(1,2)\rho(4,2)\rho(2,3)$ for large $N$.  In contrast, Theorem \ref{thm-RW} says
that the probability of observing these three points \textit{and having no other points in 
$[N{-}6,N]\times [1,7]$} is approximately $\pi(-1,2)\pi(-4,2)\pi(-2,3)$.

Our final theorem says that if we condition on the event $\{\sigma_{N-t+1}=j\}$ 
(i.e.\ $\{X^N_{(-t,j)}=1\}$) and let $N$ get large while $(N-t,j)$ remains well below the diagonal,
then the conditional distribution of points above and to the left of $(N-t,j)$ (and near $(N-t,j)$) approaches the (unconditional) distribution of points in the lower right corner of the 
square $[1,N]^2$.  

\begin{thm}\label{thrm-conditional}
Let $D$ and $F$ be disjoint finite subsets of $Q$.  
Then
\begin{align*}
    \lim_{N-t-j \rightarrow \infty} & P_N(X^N_{q+(-t,j)}=1\;\; \forall q\in D \mbox{ and }
       X^N_{r+(-t,j)}=0 \;\; \forall r\in F\,|\, X^N_{(-t,j)}=1)  \\
       = &\;  P_{\infty}(X_{q}=1\;\; \forall q\in D \mbox{ and } X_{r}=0 \;\; \forall r\in F)\,.
 \end{align*}
\end{thm}

\section{Terminology and Useful Results} 

\subsection{Pattern Avoiding Permutations and Dyck Paths}

Although we focus on the pattern 312, we first give a general definition
of pattern avoidance.

\begin{defi}
Let $k$ be a positive integer $k\geq 2$ and $\tau=\tau_1\ldots \tau_k \in S_k.$
\begin{itemize}
\item [(a)] We say that a string of $k$ distinct integers $\alpha_1\ldots\alpha_k$ forms the 
pattern $\tau$ if for each $i=1,\ldots,k$, $\alpha_i$ is the $\tau_i$th smallest element of 
$\{\alpha_1,\ldots,\alpha_k\}$. In this case we also write $\tau\,=\,\patt(\alpha_1,\ldots,\alpha_k)$. 
\item [(b)] We say that $\sigma \in S_N$ contains pattern $\tau$ if some $k$-element subsequence $\sigma_{i_1}\sigma_{i_2}\ldots\sigma_{i_k}$ of $\sigma$
occurs with the same relative order as $\tau=\tau_1\ldots \tau_k$, i.e. $\tau=\patt(\sigma_{i_1},
\ldots,\sigma_{i_k})$. If $\sigma$ does not contain the pattern $\tau$, then we 
say $\sigma$ avoids $\tau$. Let $S_N(\tau)$ be the set of permutations of $\left\{1,\ldots N \right\}$ that avoid $\tau$.
\end{itemize}
\end{defi}

\noindent
Explicitly,
a permutation $\sigma$ avoids the pattern $312$ if $\sigma$ has no subsequence of three 
elements that has same relative order as $312$, i.e., if there does not exist $i_1 < i_2 < i_3$ such
that $\sigma_{i_1}>\sigma_{i_3}> \sigma_{i_2}$.  See Figure \ref{fig:exp} for an example.
\begin{defi}\label{Dyckpath}
A Dyck segment from  $(X_0,Y_0)$ to $(X_K,Y_K)$ is a sequence $(X_0,Y_0), (X_1,Y_1),
 \ldots,$ $(X_K,Y_K)$ in $\mathbb{Z}^2$ such that and $X_0$ and $Y_0$ are nonnegative integers
 and
 \begin{equation}\label{eq:Dyck1}
X_i-X_{i-1}=1\text{, }\;
Y_i-Y_{i-1} \in \left\{ -1,+1\right\} \text{, and }\; Y_i \geq 0
\text{\hspace{3mm}for every $i=1, \ldots , K$.} 
\end{equation}
 A Dyck path of length $2L$ is a Dyck segment from $(0,0)$ to $(2L,0)$.
See Figure \ref{fig:dyckpath} below for an example. 
\end{defi}

\begin{lem}\label{CardDyckseg1}\cite{solomon}
The set $D$ of all Dyck segments from $(X_0,Y_0)$ to $(X_K,Y_K)$ has $\left|D\right|=0$ iff at least one of the following conditions holds:

\smallskip
\noindent
(a) $X_0\,>\, X_K$,

\smallskip
\noindent
(b) $|Y_K-Y_0|\, >\, X_K-X_0$, 

\smallskip
\noindent
(c) $Y_K-Y_0 \neq X_K-X_0 \, (mod \,2)$ 
(i.e., one of $Y_K-Y_0$ or $X_K-X_0$ is even and the other is odd).

\smallskip
\noindent
Otherwise, 
\begin{align*}
\left|D\right|\;=\;\binom{X_K-X_0}{\frac{X_K-X_0+\left|Y_K-Y_0 \right|}{2}}-\binom{X_K-X_0}{\frac{X_K-X_0+Y_K+Y_0+2}{2}} \;=\; 
\binom{X_K-X_0}{\frac{X_K-X_0+Y_K-Y_0}{2}}-\binom{X_K-X_0}{\frac{X_K-X_0+Y_K+Y_0+2}{2}}.
\end{align*} 
\end{lem}

\begin{rem}
\label{remDseg}
For Dyck segments from $(0,0)$ to $(X_K,Y_K)$, the result of Lemma \ref{CardDyckseg1} can be rewritten as 
\begin{align*}
\left|D\right|
\;=\;\frac{2Y_k+2}{X_k+Y_k+2}\binom{X_k}{\frac{X_k+Y_k}{2}}.
\end{align*} 
\end{rem}
\begin{defi}
$(X_i,Y_i)$ is a peak of a given Dyck path if $Y_{i-1}=Y_{i+1}=Y_i - 1$. 
\end{defi}

Krattenthaler~\cite{Kratt1} proves that there is a bijection between $S_N(132)$ and the set $D_N$ of all Dyck paths of length $2N$. We restate his result by replacing $132$-avoiding 	permutations by its complement $312$-avoiding permutations. A different bijective proof between Dyck paths and $S_N(312)$ can also be found in~\cite{bandlow}. 

	Let $\pi \in S_{N}(312)$ and $\pi=\pi_1\pi_2\ldots \pi_N$. Following the steps of Krattenthaler's proof we first determine the left-to-right maxima in $\pi$. 	A left-to-right maximum is an element $\pi_i$ which is greater than all the elements to its left, i.e., larger than all $\pi_j$ with $j<i$. For example 	left-to-right maxima in the permutation $25647318$ are $2,5,6,7$ and $8$. Let the left-to-right maxima in $\pi$ be $M_1,M_2,\dots, M_s$ so that $\pi= M_1W_1M_2W_2 \ldots M_{s}W_s$ where $W_i$ is the (possibly empty) subword of $\pi$ in between $M_i$ and $M_{i+1}$. Then any left-to-right maximum is translated into $M_i-M_{i-1}$ up-steps (with the convention that $M_0=0$). Any subword $W_i$ is translated into $\left|W_i\right|+1$ down-steps (where $\left|W_i \right|$ denotes the number of elements of $W_i$). Hence, if $\pi_i = M_t$ for some $t \leq i$, then
 the corresponding point on the Dyck path has its horizontal component equal to
$\sum_{k=1}^{t} (M_k - M_{k-1}) + \sum_{k=1}^{t-1} \left|W_k \right| +(t-1) $. 
Observe that  $\sum_{k=1}^{t-1}\left| W_k \right| $ counts the number of elements that are not a maximum up until the $t$-th maximum $M_t$, and $t-1$ counts the previous maxima $M_1,\ldots, M_{t-1}$. Together they count the number of all positions to the left of $i$, which is $i-1$, i.e., 
$\sum_{k=1}^{t-1} \left|W_k \right| +(t-1) = i-1$. Hence, the horizontal component of the point on the Dyck path corresponding to $\pi_i=M_t$ is $M_t +(i-1)$. Similarly, the vertical component of the point on the Dyck path corresponding to $\pi_i=M_t $ is $M_t-(i-1)$. Therefore, the left-to-right maximum $M_t$ at position $i$ corresponds to a peak $(M_t+(i-1),M_t-(i-1))$ in the corresponding Dyck path. 
For example, Figures \ref{fig:perm} and \ref{fig:dyckpath} show the correspondence between 
the permutation $\pi= 25647318$ and its Dyck path.  The fourth left-to-right maximum in Figure 
\ref{fig:perm}, namely 7 (circled), corresponds to the peak $(11,3)$ in Figure \ref{fig:dyckpath} 
(dashed lines).
Observe that a clockwise $45^{\circ}$ rotation of the dashed lines in Figure \ref{fig:perm} produces
the diagonal lines and horizontal axis of Figure \ref{fig:dyckpath}.  Explicitly, this rotation maps
a point $(x,y)$ in Figure \ref{fig:perm} to the point $(x+y-1,y-x+1)$ in Figure \ref{fig:dyckpath}.

	\begin{figure}[ht]
	\begin{minipage}[b]{0.45 \linewidth}
		\centering
		\begin{tikzpicture}[scale=0.5]
			\draw (0,0) rectangle (9,8);
	  	\draw [dashed] (1,0) -- (9,8); 
	  	\draw [dashed] (1,0) -- (1,2);
	  	\draw [dashed] (1,2) -- (2,2);
	  	\draw [dashed] (2,2) -- (2,5);
	  	\draw [dashed] (2,5) -- (3,5);
	  	\draw [dashed] (3,5) -- (3,6);
	  	\draw [dashed] (3,6) -- (5,6);
	  	\draw [dashed] (5,6) -- (5,7);
	  	\draw [dashed] (5,7) -- (8,7);
	  	\draw [dashed] (8,7) -- (8,7.95);
	  	\draw [dashed] (8,7.95) -- (9,7.95);	
	
		\draw [fill] (1,2) circle (0.1);
		\draw [fill] (2,5) circle (0.1);
		\draw [fill] (3,6) circle (0.1);
		\draw [fill] (4,4) circle (0.1);
		\draw [fill] (5,7) circle (0.1);
		\draw (5,7) circle (0.2);
		\draw [fill] (6,3) circle (0.1);
		\draw [fill] (7,1) circle (0.1);
		\draw [fill] (8,8) circle (0.1);

		\node [below] at (0,0) {$0$};
		\node [below] at (1,0) {$1$};
		\node [below] at (2,0) {$2$};
		\node [below] at (3,0) {$3$};
		\node [below] at (4,0) {$4$};
		\node [below] at (5,0) {$5$};
		\node [below] at (6,0) {$6$};
		\node [below] at (7,0) {$7$};
		\node [below] at (8,0) {$8$};
		\node [below] at (9,0) {$9$};
		\node [left] at (0,1) {$1$};
		\node [left] at (0,2) {$2$};
		\node [left] at (0,3) {$3$};
		\node [left] at (0,4) {$4$};
		\node [left] at (0,5) {$5$};
		\node [left] at (0,6) {$6$};
		\node [left] at (0,7) {$7$};
		\node [left] at (0,8) {$8$};
	\end{tikzpicture}
	\caption{$\pi=25647318$ with diagonal segment corresponding to x-axis in Figure \ref{fig:dyckpath}.  Here, $M_1=2$, $M_4=7$, $W_1$ is empty, and $W_4=31$.}
	\label{fig:perm}
	\end{minipage}
	\hspace{0.5cm}
	\begin{minipage}[b]{0.5\linewidth}
		\centering
		\begin{tikzpicture}[yscale=.5, xscale=.5]
		  \draw [<->] (27,0) -- (10,0) -- (10,5);
	    \draw [dashed] (10,3) -- (21,3) -- (21,0);
	    \node [left] at (10,1) {$1$};
	    \node [left] at (10,2) {$2$};
	    \node [left] at (10,3) {$3$};
	    \node [left] at (10,4) {$4$};
	    \node [below] at (10,0) {$0$};
	    \node [below] at (12,0) {$2$};
	    \node [below] at (16,0) {$6$};
	    \node [below] at (18,0) {$8$};
	    \node [below] at (21,0) {$11$};
	    \node [below] at (25,0) {$15$};
	    \draw (10,0) -- (12,2) -- (13,1) -- (16,4) -- (17,3) -- (18,4) -- (20,2) -- (21,3) -- (24,0) -- (25,1)  -- (26,0);
	  \end{tikzpicture}
	  \caption{Dyck path of length 16 corresponding to $\pi=25647318$.}
	  \label{fig:dyckpath}
	\end{minipage}
	\end{figure}

\subsection{Approximations of Integrals}

We first record  two results that will be useful in Section \ref{sec-asymp} for obtaining the 
asymptotic  behaviors of sums.   
For a function $g$, let $\|g\|_{\infty}$ be the supremum of $|g|$ over its domain.	

\begin{prop}\label{thmapprox}
	Let $f:[0,\infty) \rightarrow \mathbb{R} $ be continuous and differentiable. Let $\Delta>0$,
$R\in \mathbb{N}$, and define $X_i=i\Delta$ for $i=1,2,\ldots,R$.  Also let 
$J=\sum_{i=1}^{R}f(X_i)\Delta$ and $I=\int_{0}^{R\Delta} f(x)\,dx$.
Then 
\[
	\left| J-I \right|\;\leq \; \|f^{'} \|_{\infty} \frac{\Delta^2 R}{2}.
\]
\end{prop}

\begin{prop}\label{thmapprox2}
Let $\tilde{f}:[0,\infty)^2 \rightarrow \mathbb{R} $ be continuous and differentiable.  Let  $\Delta>0$, $R\in\mathbb{N}$, and $Y_i=X_i=i\Delta$ for $i=1,2,\ldots,R$.  
Let $\tilde{J}=\sum_{i=1}^{R}\sum_{j=1}^{R} \tilde{f}(X_i,Y_j)\Delta^2$, $\tilde{I}=\int_{0}^{R\Delta}\int_{0}^{R\Delta} \tilde{f}(x,y)\,dx\,dy$, and
$C_{\tilde{f}}:=\max\left\{ \left\|\frac{\partial \tilde{f}}{\partial x}\right\|_{\infty},
\left\|\frac{\partial \tilde{f}}{\partial y} \right\|_{\infty}\right\}$. Then
	\[\left| \tilde{J}-\tilde{I} \right|\;\leq\; C_{\tilde{f}}\Delta^3 R^2.\]
\end{prop}

In particular, suppose $\theta>0$, $R=R_N=\left\lceil N^{1/2+\theta} \right\rceil$, and $\Delta=\Delta_N = N^{-1/2}$.  Then in Proposition \ref{thmapprox}, the sums
 $I=I(N)$ and $J=J(N)$  satisfy
\[  \left| J(N)-I(N) \right| \;=\; O(N^{\theta - 1/2}), \]
and in Proposition \ref{thmapprox2} the sums $\tilde{I}=\tilde{I}(N)$ and $\tilde{J}=\tilde{J}(N)$ satisfy
\[\left| \tilde{J}(N)-\tilde{I}(N) \right| \;=\; O(N^{2\theta-1/2}).  \]

Such results are well known.  For example, 
Proposition \ref{thmapprox2} holds because
\[   \left|  \int^b_{b-\Delta}\int^a_{a-\Delta}(\tilde{f}(x,y)-\tilde{f}(a,b))\, dx\,dy\right|  \;\leq\;
   \int^b_{b-\Delta}\int^a_{a-\Delta}(|a-x|+|b-y|)\,C_{\tilde{f}}\,dx\,dy   \;=\; 
  \Delta^3 \,C_{\tilde{f}} \,.  
  \]

The following lemma is useful in the proof of Theorem \ref{propapprox1}. 
\begin{lem}\label{remintegral}
Let $T > 0$ and $K > 0$. Then 
$\int_{0}^{T}z^2e^{-Kz^2}dz=\frac{\sqrt{\pi}}{4K^{3/2}}(1+O(e^{-T\sqrt{K}})),$ 
where the term $O(e^{-T\sqrt{K}})$ is uniform over $T$ and $K$ such that $T\sqrt{K}$ is 
sufficiently large.	
\end{lem}
\begin{proof}
 We know that 
	$\int_{0}^{\infty}z^2e^{-Kz^2}dz\,=\,\sqrt{\pi}/4K^{3/2}$.
	 Also, $\int_{T}^{\infty}z^2e^{-Kz^2}dz \,=\, K^{-3/2}\int_{T\sqrt{K}}^{\infty} u^2e^{-u^2}du 
	 \,$ $\leq\, K^{-3/2}\int_{T\sqrt{K}}^{\infty} e^{2u-u^2}du$.
	Since $2u-u^2 \leq -u$ for $u \geq 3$, we have
	\begin{align*}
		\int_{T\sqrt{K}}^{\infty}e^{2u - u^2} du \; \leq\; \int_{T\sqrt{K}}^{\infty}e^{-u} du \;=\;
		 e^{-T\sqrt{K}} \hspace{4mm}\mbox{ for } T\sqrt{K} \geq 3.
	\end{align*}
	Therefore, 
	$\int_{T}^{\infty}z^2e^{-Kz^2}dz=K^{-3/2}O(e^{-T\sqrt{K}})$. We conclude that
	\begin{align*}
\int_{0}^{T}z^2e^{-Kz^2}dz=\int_{0}^{\infty}z^2e^{-Kz^2}dz-\int_{T}^{\infty}z^2e^{-Kz^2}dz=\frac{\sqrt{\pi}}{4K^{3/2}}(1+O(e^{-T\sqrt{K}})).
 \end{align*}
\end{proof}

The next result is used in the proof of Theorem \ref{propapprox3}.
\begin{lem}\label{doubleintegral}
For positive $K_1$, $K_2$, $K_3$, $w_1$, and $w_2$,  we have
\begin{align*}
   \int_{0}^{w_2}\int_{0}^{w_1}xy\, & e^{-K_1x^2}e^{-K_2y^2}\left(e^{-K_3(y-x)^2}-e^{-K_3(x+y)^2}
\right)\,dx\,dy   \\
    &=\; \frac{\pi K_3}{4(K_1K_2+K_1K_3+K_2K_3)^{3/2}} \,+\,
    O\left(\frac{e^{-K_1w_1^2}+e^{-K_2w_2^2}}{K_1K_2}\right) \,.  
    \end{align*}
\end{lem}
\begin{proof}
Using standard properties of bivariate Gaussian integrals, 
we know that for positive $A$ and $B$ and real  $C$ such that $4AB-C^2 > 0$, 
\[  \int_{-\infty}^{\infty}\int_{-\infty}^{\infty}xye^{-(Ax^2+By^2+Cxy)}\,dx\,dy\;=\;
    \frac{-2\pi C}{(4AB-C^2)^{3/2}}.\]
Letting $A=K_1+K_3$, $B=K_2+K_3$ and $C= \pm\,2K_3$, we obtain
\begin{equation}
  \nonumber
  \int_{-\infty}^{\infty}\int_{-\infty}^{\infty}xye^{-K_1x^2}e^{-K_2y^2}e^{-K_3(x\pm y)^2} \,dx\,dy
  \;=\; \frac{\mp\,\pi K_3}{2(K_1K_2+K_1K_3+K_2K_3)^{3/2}}.
  \end{equation}
Let  $h(x,y)\,=\,xy\,e^{-K_1x^2}e^{-K_2y^2}(e^{-K_3(y-x)^2}-e^{-K_3(x+y)^2})$.  Since $h(x,y)$ is an even function of $x$ and of $y$, we have 
\begin{equation}
\label{eq.quad2}
\int_{0}^{\infty}\int_{0}^{\infty}h(x,y)\,dx\,dy
\;=\;\frac{1}{4}\int_{-\infty}^{\infty}\int_{-\infty}^{\infty}h(x,y)\,dx\,dy
\;=\; \frac{\pi K_3}{4(K_1K_2+K_1K_3+K_2K_3)^{3/2}}.
\end{equation}
For $a,b\geq 0$, we also have the simple bounds
\begin{equation}
\label{eq.quad3}
0  \;<\; \int_b^{\infty}\int_a^{\infty} h(x,y)\, dx\,dy 
\;< \; \int_b^{\infty}\int_a^{\infty} xy\,e^{-K_1x^2}e^{-K_2y^2}\,dx\,dy
 \;=\;\frac{e^{-(K_1a^2+K_2b^2)}}{4K_1K_2}.
\end{equation}
Using Equation (\ref{eq.quad3}), we see that
\[
   \left| \int_0^{w_2}\int_0^{w_1} h  \,-\, \int_0^{\infty}\int_0^{\infty} h\right|  \;\leq \;
      \int_0^{\infty}\int_{w_1}^{\infty}h \,+\, \int_{w_2}^{\infty}\int_0^{\infty}h  \;=\;
      O\left(\frac{e^{-K_1w_1^2}+e^{-K_2w_2^2}}{K_1K_2}\right). 
\]
The lemma follows from this and Equation (\ref{eq.quad2}).
\end{proof}

\section{Proofs of the Exact Results}

\begin{defi}\label{defiinsert}
	Let $\sigma$ and $\tau$ be permutations of lengths $N$ and $M$ respectively.
For $i\in [1,N]$, we define $\ins(\tau,\sigma,i)$ (see Figure \ref{fig:insert1}) to be 
the permutation $\theta$ in $S_{N+M}$ given by
	$$\theta_k=\left\{ \begin{array}{cll}
		\sigma_k			&\mbox{ if } k < i \mbox{ and } \sigma_k < \sigma_i ,\\ 
		\sigma_k+M                  &\mbox{ if } k < i \mbox{ and } \sigma_k > \sigma_i ,\\  
		\tau_{k-i+1}+\sigma_i		&\mbox{ if } i\leq k < i+M, \\
		\sigma_{k-M}		&\mbox{ if } k \geq i+M \mbox{ and } \sigma_{k-M} \leq \sigma_i ,\\
		\sigma_{k-M}+M	&\mbox{ if } k\geq i+M \mbox{ and } \sigma_{k-M}>\sigma_i.
	\end{array}\right.$$
\end{defi}

\begin{center}
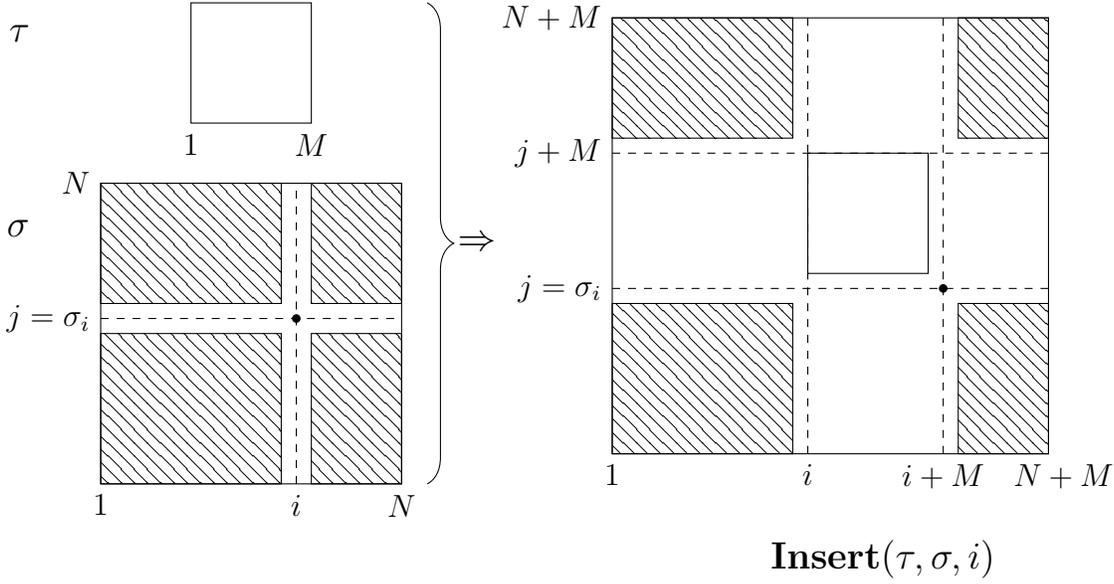
\begin{figure}[htp]
\begin{tikzpicture}[scale=.20]
	\draw (0,0) rectangle (20,20);
  \draw [dashed] (0,11) -- (20,11);
  \draw [dashed] (13,0) -- (13,20);
  \draw [fill] (13,11) circle (0.25);
  \node [below] at (0,0) {1};
  \node [left] at (0,11) {$j=\sigma_i$};
  \node [left] at (0,20) {$N$};
  \node [below] at (13,0) {$i$};
  \node [below] at (20,0) {$N$};
  \node [left] at (-4,17) {\large $\sigma$};
  \draw [pattern=custom north west lines] (0,0) rectangle (12,10);     
  \draw [pattern=custom north west lines] (14,0) rectangle (20,10);    
  \draw [pattern=custom north west lines] (0,12) rectangle (12,20);    
  \draw [pattern=custom north west lines] (14,12) rectangle (20,20);   

  \draw (6,24) rectangle (14,32);
  \node [below] at (6,24) {$1$};
  \node [below] at (14,24) {$M$};
  \node [left] at (-4,30) {\large $\tau$};

  \draw [decorate,decoration={brace,amplitude=10pt,mirror,raise=4pt},yshift=0pt]
  (21,0) -- (21,32) node [black,midway,xshift=0.8cm] {\large $\Rightarrow$};

  \draw (34,2) rectangle (63,31);
  \draw [dashed] (34,13) -- (63,13); 
  \draw [dashed] (34,22) -- (63,22);
  \draw [dashed] (47,2) -- (47,31); 
  \draw [dashed] (56,2) -- (56,31);
  \draw [fill] (56,13) circle (0.25);
  \node [below] at (34,2) {1};
  \node [left] at (34,13) {$j=\sigma_i$};
  \node [left] at (34,22) {$j+M$};
  \node [left] at (34,31) {$N+M$};
  \node [below] at (47,2) {$i$};
  \node [below] at (56,2) {$i+M$};
  \node [below] at (64,2) {$N+M$};
  \node [below] at (52,-3) {\large $\ins(\tau,\sigma,i)$};

  \draw [pattern=custom north west lines] (34,2) rectangle (46,12);     
  \draw [pattern=custom north west lines] (57,2) rectangle (63,12);     
  \draw [pattern=custom north west lines] (34,23) rectangle (46,31);     
  \draw [pattern=custom north west lines] (57,23) rectangle (63,31);     
  \draw (47,14) rectangle (55,22);	

\end{tikzpicture}
\caption{Diagram of how the graphs of $\tau$ and $\sigma$ combine to make the graph
of $\ins(\tau,\sigma,i)$.  The graph of $\sigma$ is broken into four rectangles, which 
are then moved apart to make room for $\tau$.} 
\label{fig:insert1}
\end{figure}
\end{center}

The next result shows that the Insert operation preserves $312$-avoidance in a certain 
situation.

\begin{prop}\label{propS_A(N+M,i+M,j)}
Let $\sigma \in S^{\Box}(N,i,j)$ and $\tau \in S_{M}(312)$. 
Then $\ins(\tau,\sigma,i) \in S^{\bullet}(N+M,i+M,j)$.
\end{prop} 

\begin{proof}
	Let $\theta= \ins(\tau,\sigma,i)$.
Since $\sigma_i=j$, Definition \ref{defiinsert} implies that 	
$\theta_{i+M}=j$.  
We must now
show that $\patt(\theta_{k_1}\theta_{k_2}\theta_{k_3}) \neq 312$
whenever $1\leq k_1 < k_2 < k_3 \leq N+M$.  It is not hard to
do this by checking the cases that $k_1$ is in $[1,i)$, $[i,i+M)$,
or $[i+M,N+M]$.  We leave this to the reader, with the aid of
Figure \ref{fig:insert1}.
\end{proof}

The next lemma may be viewed as a converse of Proposition \ref{propS_A(N+M,i+M,j)}.
Roughly speaking, part (f) shows that every element of $S^{\bullet}(\cdot,\cdot,\cdot)$ may be expressed
as the result of an Insert operation.  Lemma \ref{lemonetoone} then shows that such an 
expression is unique.  This will permit us to evaluate the cardinality of $S^{\bullet}(\cdot,\cdot,\cdot)$
by using the Insert operation to construct an explicit bijection.

\begin{lem}\label{lemmain}
Let $\sigma \in S^{\bullet}(M,M-t,j)$ with $j < M-t$. Let $i_0=\min \left\{i \in [1,M]: \sigma_i > j \right\}$. Then \\
(a) $i_0 \in [1,M-t)$;  \\
(b) $\sigma_i > j$ for every $i \in [i_0,M-t)$;
\\
(c) $\sigma_i \in [1,j)$ for every  $i \in [1,i_0)$;
\\
(d) $j \geq i_0 \geq \max \left\{1,j-t \right\}$;  \\
(e) The image of the domain $[i_0,M-t)$ under $\sigma$ equals $(j,j+M-t-i_0]$;  \\
(f) Let $\hat{\sigma}=\patt(\sigma_1,\ldots,\sigma_{i_0-1},\sigma_{M-t},\ldots,\sigma_{M})$ and $\tilde{\sigma}=\patt(\sigma_{i_0},\ldots,\sigma_{M-t-1})$. Then $\hat{\sigma} \in S^{\Box}(t+i_0,i_0,j)$, $\tilde{\sigma} \in S_{M-t-i_0}(312)$ and $\sigma=\ins(\tilde{\sigma},\hat{\sigma},i_0)$.
\end{lem}

	Figure \ref{fig:1} depicts the properties of $\sigma \in S^{\bullet}(M,M-t,j)$ that are  stated in Lemma \ref{lemmain}.

\begin{proof}
	Let $\sigma \in S_M(312)$ such that $\sigma_{M-t}=j$. 

\begin{figure}[htp]
\begin{tikzpicture}[scale=.5]
    \draw [<->] (-5,13) -- (-5,-5) -- (13,-5); 

    \draw [pattern=custom north west lines] (-5,-5) rectangle (-1,0);
    \draw (0,2) rectangle (6,8); 
    \draw [pattern=custom north west lines] (8,-5) rectangle (12,0); 
    \draw [pattern=custom north west lines] (8,9) rectangle (12,12); 

    \draw [fill] (7,1) circle (.1);

    \node [below] at (-1.3,-5) {$i_0-1$};
    \node [below] at (0.2,-5) {$i_0$};
    \draw [dashed] (0,-5) -- (0,2) ;
    \node [below] at (7,-5) {$M-t$};
    \draw [dashed] (7,1) -- (7,-5);
    \draw [dashed] (8,-5) -- (8,12);
    \node [below] at (12,-5) {$M$};
    \draw [dashed] (12,-5) -- (12,12);

    \node [left] at (-5,0) {$j-1$};
    \draw [dashed] (-5,0) -- (12,0);
    \node [left] at (-5,1) {$j$};
    \draw [dashed] (-5,1) -- (7,1);
    \node [left] at (-5,2) {$j+1$};
    \draw [dashed] (-5,2) -- (0,2);
    \node [left] at (-5,8) {$j+M-t-i_0$};
    \draw [dashed] (-5,8) -- (0,8);
    \draw [dashed] (-5,9) -- (12,9);
    \node [left] at (-5,12) {$M$};
    \draw [dashed] (-5,12) -- (12,12);
    \node at (3,5) {\LARGE$\tilde{\sigma}$};

\end{tikzpicture}
	\caption{Illustration of properties described in Lemma \ref{lemmain} for a permutation 
	$\sigma \in S_M{(312)}$ with $\sigma_{M-t}=j$.  Except for the point $(M-t,j)$, all points of 
	the graph of $\sigma$ are inside one of the four rectangles bounded by solid lines.}
	\label{fig:1}
\end{figure}
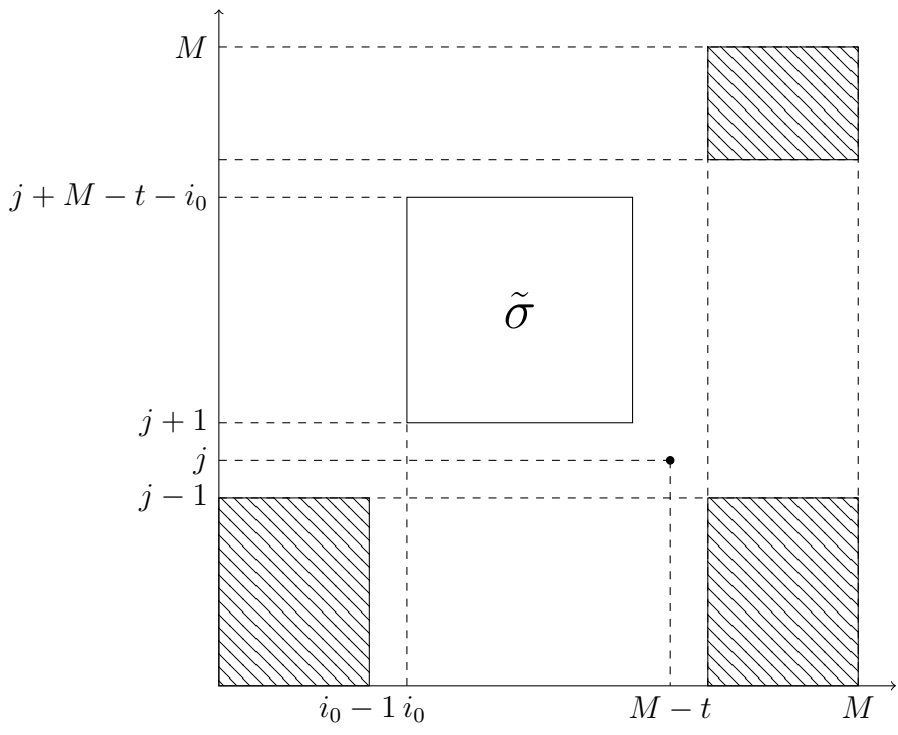

\smallskip\noindent
(a) $j < M-t$ implies that there is at least one element $i \in \left[ 1, M-t \right)$ such that $\sigma_i > j$. Hence, $i_0 \in \left[ 1, M-t \right)$.

\smallskip\noindent
(b)  If this were false, then there would exist $ i\in (i_0,M-t) $ such that $\sigma_i=j_i< j <\sigma_{i_0}$.  But then we would have $\patt(\sigma_{i_0},\sigma_{i},\sigma_{M-t})\,=\,312$. 

\smallskip\noindent
(c) The definition of $i_0$ and the fact that $M-t \notin [1,i_0)$ imply that $\forall i \in [1,i_0)$, $\sigma_{i} \in [1,j)$.

\smallskip\noindent
(d)
Part (c) and the Pigeonhole Principle imply that $i_0 \leq j$. Once $i_0 -1$ elements are mapped 
to the range $[1,j)$, the remaining $j-i_0$ elements in $[1,j)$ should be the images of elements 
in the domain $(M-t,M]$ (using part (b) and $\sigma_{M-t}=j$). Hence, $j-i_0 \leq |(M-t,M]| = t$ 
and therefore $i_0 \geq \max \left\{1,j-t\right\}$.  

\smallskip\noindent
(e) 
Assume $\exists$ $ I \in \left[i_0,M-t \right)$ such that $\sigma_{I}= J > j+M-t-i_0$. This implies that there is an element $N$ in $(j,j+M-t-i_0 ]$ that is not the image of any element in $[i_0,M-t)$. Hence, there is an element $K$ in $(M-t, M]$ such that 
$\sigma_{K} = N$. Therefore, $\sigma_{I}\sigma_{M-t}\sigma_{K}=JjN$ is a 312-pattern. This contradiction, together with part (b), proves part (e).

\smallskip\noindent
(f)  Since $\sigma$ avoids 312, clearly $\hat{\sigma}\in S_{t+i_0}(312)$ and 
$\tilde{\sigma} \in S_{M-t-i_0}(312)$.   By part (b), $\hat{\sigma}_{i_0}=\sigma_{M-t}=j$.  
Also, for all $i< i_0$, $\hat{\sigma}_{i}=\sigma_i <j$ by part (c). Hence, $\hat{\sigma} \in S^{\Box}(t+i_0,i_0,j)$.   Finally, from part (e), we deduce that $\sigma_i=\tilde{\sigma}_{i-i_0+1}+j$ for 
all $i \in [i_0,M-t)$. Using this and part (c) 
we conclude that $\sigma=\ins(\tilde{\sigma},\hat{\sigma},i_0)$.
\end{proof}

\begin{lem}\label{lemonetoone}
	Fix $N$, $t$, and $j$.  Assume that $\ins(\tau^0,\rho^0,i^0)=\ins(\tau^1,\rho^1,i^1)$, where $\tau^{k} \in S_{N-t-i^k}(312)$ and $\rho^k \in S^{\Box}(t+i^k,i^k,j)$ for $k=0,1$. 
	Then $i^0=i^1$, $\tau^0=\tau^1$ and $\rho^0=\rho^1$. 
\end{lem}

\begin{proof}
	If $i^0=i^1$ it is easy to see that $(\tau^0,\rho^0)=(\tau^1,\rho^1)$. 
	Assume that $i^0<i^1$ (or similarly, we could consider $i^0 > i^1$). By its definition, $(\ins(\tau^0,\rho^0,i^0))_{i^0}=\tau^0_1+\rho^0_{i^0}=\tau^0_1+j > j$. On the other hand, $(\ins(\tau^1,\rho^1,i^1))_{i^0}=\rho^1_{i^0} <j$ since $\rho^1 \in S^{\Box}(t+i^1,i^1,j)$. This gives a contradiction and hence we conclude the result. 
\end{proof}

Proposition \ref{propS_A(N+M,i+M,j)} allows us to make the following definition.

\begin{defi}\label{phiA}
Fix $N$, $t$, and $j$.  Let 
\begin{align*}  \dom_{\bullet}(N,t,j)\;=\bigcup_{i_0=\max\left\{1,j-t\right\}}^{j} S_{N-t-i_0}(312)\times S^{\Box}(t+i_0,i_0,j). 
\end{align*}
We define the map $\phi_{\bullet;N,t,j}:\dom_{\bullet}(N,t,j) \rightarrow S^{\bullet}(N,N{-}t,j)$ by
	\begin{align*}
		\phi_{\bullet}((\tilde{\sigma},\hat{\sigma}))=\ins(\tilde{\sigma},\hat{\sigma},i_0) 
		\hspace{5mm}\mbox{ for } 
		(\tilde{\sigma},\hat{\sigma}) \in S_{N-t-i_0}(312)\times S^{\Box}(t+i_0,i_0,j).
	\end{align*}
\end{defi}

\begin{lem}\label{lemphiA}
Fix $N$, $t$, and $j$ with $j<N{-}t$. 
Let $\phi_{\bullet;N,t,j}:\dom_{\bullet}(N,t,j) \rightarrow S^{\bullet}(N,N{-}t,j)$ be the map in Definition \ref{phiA}. 
Then $\phi_{\bullet;N,t,j}$ is a bijective map. 
\end{lem}

\begin{proof}
Assume that $(\tilde{\sigma}^0,\hat{\sigma}^0) \in S_{N-t-i_0}(312)\times S^{\Box}(t+i_0,i_0,j)$ and $(\tilde{\sigma}^1,\hat{\sigma}^1) \in S_{N-t-i_1}(312)\times S^{\Box}(t+i_1,i_1,j)$ and that 	
$\phi_{\bullet}((\tilde{\sigma}^0,\hat{\sigma}^0))=\phi_{\bullet}((\tilde{\sigma}^1,\hat{\sigma}^1))$. 
	We first apply Lemma \ref{lemonetoone} with $i^k=i_{k}$,$\tau^k =\tilde{\sigma}^{k}$, $\rho^k=\hat{\sigma}^{k}$ for $k=0,1$ and conclude that $\phi_{\bullet}$ is one-to-one.
Next we apply Lemma \ref{lemmain}(f), substituting $M=N$, and conclude that for each 
$\sigma \in S^{\bullet}(N,N-t,j)$ there exists an $i_0$ such that $\max\left\{1,j-t \right\} \leq i_0 \leq j$ and
that  $\sigma=\ins(\tilde{\sigma},\hat{\sigma},i_0)$ where
  $\hat{\sigma}=\patt(\sigma_1,\ldots,\sigma_{i_0-1},\sigma_{N-t},\ldots,\sigma_{N}) \in 
  S^{\Box}(t+i_0,i_0,j)$ and $\tilde{\sigma}=\patt(\sigma_{i_0},\ldots,\sigma_{N-t-1}) \in S_{N-t-i_0}(312)$.
Therefore, $\phi_{\bullet;N,t,j}$ is a surjective map and hence a bijection.   
\end{proof}

It is  now straightforward to prove Theorem \ref{propS_A(N,N-t,j)}.

\begin{proof}[Proof of Theorem \ref{propS_A(N,N-t,j)}]
By Lemma \ref{lemphiA}, we have
\[
		\left| S^{\bullet}(N,N-t,j) \right|\;=\; |\dom_{\bullet}(N,t,j)| \;=\;
		\sum_{i_0=\max\left\{1,j-t\right\}}^{j}C_{N-t-i_0}\left| S^{\Box}(t+i_0,i_0,j) \right|.
\]
Hence, Remark \ref{remMadras} completes the proof of Theorem \ref{propS_A(N,N-t,j)}.
\end{proof}

Next we look at cardinality of the set of $312$-avoiding permutations that has $k$ decreasing points below the diagonal.  That is, for  $j_k<j_{k-1} < \ldots < j_1 < N-t_1 < \ldots < N-t_k$,
we shall prove the recursive formula  in Theorem \ref{propS_B(N)}  for the cardinality of
\[		S^{{\searrow}_k}(N)\; \equiv\;  S^{{\searrow}_k}(N,N-t_1,\ldots,N-t_k,j_1,\ldots,j_k)=
		\left\{\sigma \in S_N(312):\sigma_{N-t_1}=j_1, \,\ldots \,\sigma_{N-t_k}=j_k \right\}.
\]

\begin{proof}[Proof of Theorem \ref{propS_B(N)}]
Let $\sigma \in S^{{\searrow}_k}(N)$.  Since $S^{{\searrow}_k}(N) \subseteq S^{\bullet}(N,N-t_{k},j_{k})$, we apply Lemma \ref{lemmain} with $M=N$, $t=t_{k}$, $j=j_{k}$ and writing $i$ instead of $i_0$. See Figure \ref{fig-S_B(N)}.  
By part (f) of this Lemma we conclude that $\sigma=\ins(\tilde{\sigma},\hat{\sigma},i)$, where 
$\hat{\sigma}=\patt(\sigma_1,\ldots,\sigma_{i-1},\sigma_{N-t_{k}},\ldots,\sigma_N) \in S^{\Box}(t_{k}+i,i,j_{k})$ and $\tilde{\sigma}=\patt(\sigma_{i},\ldots,\sigma_{N-t_{k}-1}) \in S_{N-t_{k}-i}(312)$. 
By Lemma \ref{lemmain}(d) we also know that 
$\displaystyle \max\left\{1,j_{k}-t_{k} \right\} \leq i \leq j_{k}$. Since 
$ i \leq j_{k}< \ldots <j_1< N{-}t_1<\ldots < N{-}t_{k}$ and using Lemma \ref{lemmain}(e), 
we see that $\tilde{\sigma}_{N-t_r-i+1}=\sigma_{N-t_r}-\sigma_{N-t_{k}}=j_r-j_{k}$ for 
$r=1,\ldots,k{-}1$.
Hence, $\tilde{\sigma} \in S^{{\searrow}_{k-1}}(N-t_{k}-i,N-t_1-(i-1),N-t_2-(i-1),\ldots,N-t_{k-1}-(i-1),j_1-j_{k},\ldots,j_{k-1}-j_{k})$.

\begin{figure}[htp]
\begin{center}
\begin{tikzpicture}[scale=.5]
    \draw [<->] (-5,13) -- (-5,-5) -- (13,-5); 

    \draw [pattern=custom north west lines] (-5,-5) rectangle (-3,-2);
    \draw (-2.5,-1.5) rectangle (8.5,9.5); 
    \draw [pattern=custom north west lines] (9,-5) rectangle (12,-2); 
    \draw [pattern=custom north west lines] (9,9.75) rectangle (12,12); 

    \draw [fill] (9,-2) circle (.12);
    \draw [fill] (7,0) circle (.1);
    \draw [fill] (6.5,0.8) circle (.07);
    \draw [fill] (5.9,1.4) circle (.07);
    \draw [fill] (5.6,2) circle (.07);
    \draw [fill] (5,3) circle (.1);

    \node [below] at (-2.5,-5) {$i$};
    \draw [dashed] (-2.5,-5) -- (-2.5,-1.5) ;
    \node [below] at (5,-5) {$N-t_1$};
    \draw [dashed] (5,3) -- (5,-5);
    \node [below] at (9,-5) {$N{-}t_k$};
    \node [below] at (12,-5) {$N$};
    \draw [dashed] (12,-5) -- (12,12);

    \node [right] at (12.2,-2) {$j_k$};
    \draw [dashed] (7,0) -- (12,0);
    \node [right] at (12.2,0) {$j_{k-1}$};
    \draw [dashed] (5,3) -- (12,3);
    \node [right] at (12.2,3) {$j_{1}$};

    \node [left] at (-5,12) {$N$};
    \draw [dashed] (-5,12) -- (12,12);
    \node at (2,0.5) {\LARGE$\tilde{\sigma}$};

\end{tikzpicture}
\end{center}
        \caption{\label{fig-S_B(N)} Proof of Theorem \ref{propS_B(N)}.  The shaded regions 
correspond to the permutation $\hat{\sigma}$.
}
\end{figure}
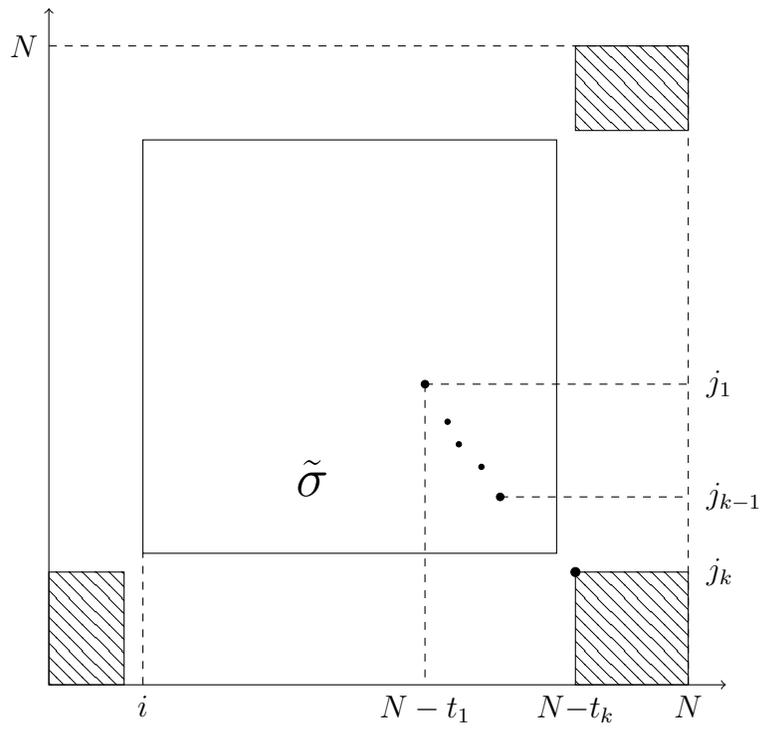

	Let 
\begin{align*}   
\dom_{\searrow_{k}} \,&\equiv\, \dom_{\searrow_{k}}(N,j_1,j_2,\ldots,j_{k},t_1,t_2,\ldots t_{k})\, 
 =\bigcup_{i=\max\left\{{1,j_{k}-t_{k}}\right\}}^{j_{k}} \\
&   \!\! S^{{\searrow}_{k-1}}(N{-}t_{k}{-}i,N{-}t_1{-}i{+}1,\ldots,N{-}t_{k-1}{-}i{+}1,j_1{-}j_k, \ldots
  ,j_{k-1}{-}j_{k})   \times S^{\Box}(t_{k}{+}i,i,j_{k}).
\end{align*} 
Observe that $\dom_{\searrow_{k}}\subset \dom_{\bullet}(N,t_{k},j_{k})$.
We define the map $\phi_{\searrow_{k}}$ 
such that 
	$\phi_{\searrow_{k}}:\dom_{\searrow_{k}} \rightarrow S^{{\searrow}_k}(N)$ is the restriction of $\phi_{\bullet;N,t_{k},j_{k}}$ to $\dom_{\searrow_{k}}$.

	In the discussion of the first paragraph, we see that $(\tilde{\sigma},\hat{\sigma}) \in \dom_{\searrow_{k}}$, and that $\sigma=\ins(\tilde{\sigma},\hat{\sigma},i)=\phi_{\searrow_{k}}(\tilde{\sigma},\hat{\sigma})$. Hence 
$\phi_{\searrow_{k}}$ is surjective, and $\phi_{\searrow_{k}}$ is injective (since by Lemma \ref{lemphiA} $\phi_{\bullet;N,t_{k},j_{k}}$ is), so $\phi_{\searrow_{k}}$ is a bijection. This implies that
	\begin{align*}
		\left|S^{{\searrow}_k}(N)\right|
		\;= & \sum_{i=\max\left\{1,j_{k}-t_{k}\right\}}^{j_{k}} 
		  \left|S^{\Box}(t_{k}+i,i,j_{k})\right| \times \\  
	&\hspace{15mm} \left|S^{{\searrow}_{k-1}}(N{-}t_{k}{-}i,N{-}t_1{-}i{+}1,\ldots,
	N{-}t_{k-1}{-}i{+}1,j_1{-}j_k, \ldots,j_{k-1}{-}j_{k})\right|.
	\end{align*}
This proves the recursion.
The formula for the case $k=2$ follows, using Theorem \ref{propS_A(N,N-t,j)}.
\end{proof}

Now we turn to the proof of Theorem \ref{propS_C(N)}, which establishes the cardinality of
\[ S^{\nearrow}(N)\equiv S^{\nearrow}(N,N-t_1,N-t_2,j_1,j_2)=\left\{ \sigma \in S_N(312):\sigma_{N-t_1}=j_1,
       \sigma_{N-t_2}=j_2 \right\} 
\]
for the situation $j_1 < N-t_1<N-t_2$, $j_1<j_2 < N-t_2$.  

\begin{proof}[Proof of Theorem \ref{propS_C(N)}]
	Let $\sigma \in S^{\nearrow}(N)$. For every $i \in \left[1,N-t_1\right)$, $\sigma_i$ cannot 
belong to the interval $\left(j_2,N\right]$ (otherwise  $\sigma_i\sigma_{N-t_1}\sigma_{N-t_2}$ 
will form a $312$ pattern). This implies that only domain elements in 
$\left(N-t_1,N\right]\setminus\left\{N-t_2\right\}$ will 	be mapped into $\left(j_2,N\right]$. Hence 
$\left|\left(j_2,N\right]\right| \leq \left|\left(N-t_1,N\right]\setminus\left\{N-t_2\right\}\right|$, 
which says $j_2 \geq N-t_1+1$. This proves part (a).

	For the rest of the proof we assume that $j_2 \geq N-t_1+1$.
	Let $\sigma \in S^{\nearrow}(N)$. 
	Observing that $S^{\nearrow}(N) \subseteq S^{\bullet}(N,N-t_2,j_2)$, we apply Lemma \ref{lemmain} with $M=N$, $t=t_2$, $j=j_2$, and writing $i_2$ for $i_0$.
	In Lemma \ref{lemmain}(d,f), we get that $j_2 \geq i_2 \geq j_2-t_2$,  and that $\sigma = \ins(\tilde{\sigma},\hat{\sigma},i_2)$ where 
$\hat{\sigma} \,=\, \patt(\sigma_1,\ldots,\sigma_{i_2-1},\sigma_{N-t_2},\dots,\sigma_N)
\in S^{\Box}(t_2+i_2,i_2,j_2)$ and  $ \tilde{\sigma}\,=\, \patt(\sigma_{i_2},\ldots,\sigma_{N-t_2-1}) \in S_{N-t_2-i_2}(312)$. 

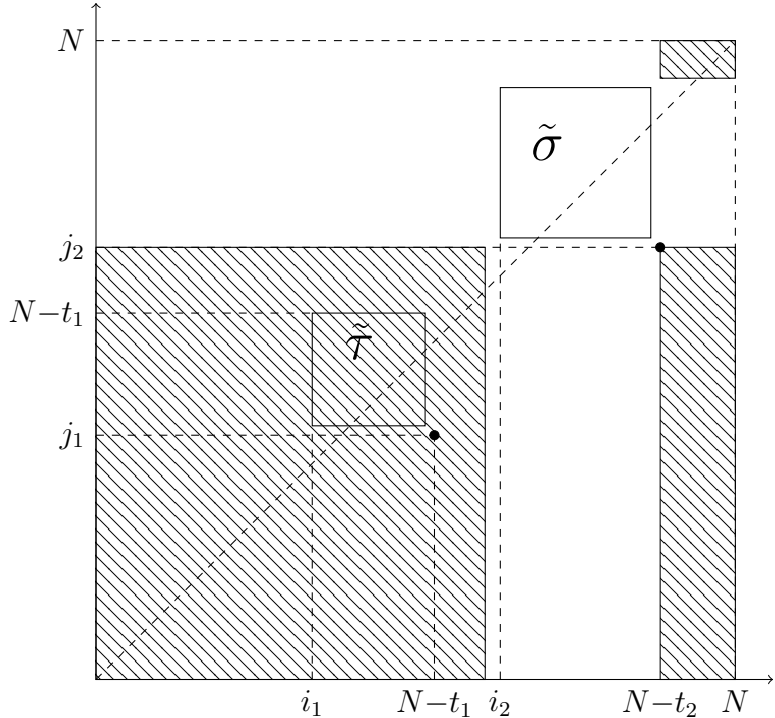
\begin{figure}[htp]
\begin{center}
\begin{tikzpicture}[scale=.5]
    \draw [<->] (-5,13) -- (-5,-5) -- (13,-5); 

    \draw [pattern=custom north west lines] (-5,-5) rectangle (5.35,6.5);
    \draw (5.75,6.75) rectangle (9.75,10.75); 
    \draw (0.75,1.75) rectangle (3.75,4.75); 
    \draw [pattern=custom north west lines] (10,-5) rectangle (12,6.5); 
    \draw [pattern=custom north west lines] (10,11) rectangle (12,12); 
    \draw [dashed] (-5,-5) -- (12,12);
    
    \draw [fill] (10,6.5) circle (.12);
    \draw [fill] (4,1.5) circle (.12);
    \node at (7,9.3) {\LARGE$\tilde{\sigma}$};
    \node at (2,4) {\LARGE$\tilde{\tau}$};
    
    \node [below] at (5.75,-5) {$i_2$};
    \draw [dashed] (5.75,-5) -- (5.75,6.75) ;
    \node [below] at (4,-5) {$N{-}t_1$};
    \draw [dashed] (4,-5) -- (4,1.5) ;
    \node [below] at (0.75,-5) {$i_1$};
    \draw [dashed] (0.75,-5) -- (0.75,1.75) ;
    \node [below] at (10,-5) {$N{-}t_2$};
    \node [below] at (12,-5) {$N$};
    \draw [dashed] (12,-5) -- (12,12);

    \node [left] at (-5,1.5) {$j_1$};
    \draw [dashed] (-5,1.5) -- (4,1.5);
    \node [left] at (-5,4.75) {$N{-}t_1$};
    \draw [dashed] (-5,4.75) -- (0.75,4.75);
    \node [left] at (-5,6.5) {$j_{2}$};
     \draw [dashed] (-5,6.5) -- (10,6.5);
    \node [left] at (-5,12) {$N$};
    \draw [dashed] (-5,12) -- (12,12);

\end{tikzpicture}
\end{center}
        \caption{\label{fig-S_C(N)}  Proof of Theorem \ref{propS_C(N)}(b).  The shaded regions 
correspond to the permutation $\hat{\sigma}$.
The diagonal is drawn for reference.
}
\end{figure}

	By Lemma \ref{lemmain}(b), $\sigma_i > j_2$ for all $i \in [i_2,N-t_2)$, and since 
$\sigma_{N-t_1}=j_1<j_2$ it follows that $N-t_1<i_2$ and hence 
$\hat{\sigma}_{N-t_1}=\sigma_{N-t_1}=j_1$. 
We conclude that  $\hat{\sigma} \in S^{\bullet}(t_2+i_2,N-t_1,j_1)$. 
Also, we have shown $i_2 \in \left[\max\left\{N-t_1+1,j_2-t_2 \right\},j_2\right]$.	
Now we will apply Lemma \ref{lemmain} one more 
time to $\hat{\sigma}$ instead of $\sigma$. Here we take $M=t_2+i_2$, $t=(t_2+i_2)-(N-t_1)$, $j=j_1$ and $i_0=i_1$. By Lemma \ref{lemmain} we get that 	
$i_1 \in [\max\left\{1,j_1-t_1-t_2-i_2+N \right\},j_1]$ and that 
$\hat{\sigma}=\ins(\tilde{\tau},\hat{\tau},i_1)$ with $\tilde{\tau} \in S_{N-t_1-i_1}(312)$ and 
$\hat{\tau} \in S^{\Box}(t_1+t_2+i_1+i_2-N,i_1,j_1)$. 
Furthermore, since $\hat{\sigma} \in S^{\Box}(t_2+i_2,i_2,j_2)$,  we know that 
$\hat{\sigma}_i < j_2$ $\forall i \in [i_1,N-t_1)\subset [1,i_2)$. Hence, by Lemma \ref{lemmain}(e) we conclude that $j_1 +(t_2+i_2)-(t_2+i_2-N+t_1)-i_1 <j_2$, i.e.$, j_1-j_2+N-t_1 < i_1$.

Let $M^*\,=\,N-t_1-i_1$ (the size of $\tilde{\tau}$).  We shall now show that
$\hat{\tau} \in S^{\Box}(i_1+i_2+t_1+t_2-N,i_2-M^*,j_2-M^*)$.  
We know $\hat{\sigma}_{i_2}=j_2$
and $\hat{\sigma}_i<j_2$ for all $i<j_2$.  We also know $i_2>N-t_1=i_1+M^*$, so by the 
definition of $\ins$ we see that either 
\begin{verse}
   ($\alpha$)  $\hat{\sigma}_{i_2}\,=\, \hat{\tau}_{i_2-M^*}$ and 
        $\hat{\tau}_{i_2-M^*}\,\leq \,\hat{\tau}_{i_1}$, \hspace{3mm}or    \\
  ($\beta$) $\hat{\sigma}_{i_2}\,=\, \hat{\tau}_{i_2-M^*}+M^*$ and 
        $\hat{\tau}_{i_2-M^*}\,> \,\hat{\tau}_{i_1}$.
\end{verse}
Now, $\hat{\tau}_{i_1}\,=\,j_1\,<\,j_2\,=\,\hat{\sigma}_{i_2}$, so ($\alpha$) does not hold.
Therefore ($\beta$) holds, so 
$\hat{\tau}_{i_2-M^*}\,=\,\hat{\sigma}_{i_2}-M^*\,=\,j_2-M^*$.  It remains to show that
$\hat{\tau}_u<j_2-M^*$ for all $u<i_2-M^*$.  On the one hand, if $u<i_1$, then
$\hat{\tau}_u\leq \hat{\sigma}_u<j_1<j_2-M^*$ (by the last inequality of the preceding
paragraph).  On the other hand, if $i_1\leq u<i_2-M^*$, then 
$\hat{\tau}_u \leq \max \left\{ j_1, \hat{\sigma}_{u+M^*}-M^* \right\}  < j_2-M^*$.
This completes the proof that
 $\hat{\tau}\in S^{\Box}(i_1+i_2+t_1+t_2-N,i_2-M^*,j_2-M^*)$.

	Let
\begin{align*}
		&\dom_{\nearrow}\; \equiv\; \dom_{\nearrow}(N,j_1,j_2,t_1,t_2)\;= \\
		& \bigcup_{i_1=\max\left\{1,j_1-j_2+N-t_1+1 \right\}}^{j_1}\bigcup_{i_2=\max\left\{N-t_1+1,j_2-t_2 \right\}}^{j_2} S_{N-t_2-i_2}(312) \times 	S_{N-t_1-i_1}(312)\times S_{\cap}(N,i_1,i_2)
	\end{align*}
where
\[
	S_{\cap}(N,i_1,i_2)= S^{\Box}(i_1+i_2+t_1+t_2-N,i_1,j_1) \cap S^{\Box}(i_1+i_2+t_1+t_2-N,i_2-M^*,j_2-M^*) \,.
\]
We define the map $\phi_{\nearrow} \equiv \phi_{\nearrow;N,t_1,t_2,j_1,j_2}:\dom_{\nearrow} \rightarrow S^{\nearrow}(N)$ by
	\begin{align*}
		\phi_{\nearrow}(\alpha,\beta,\gamma) \;=\;& \ins(\alpha,\ins(\beta,\gamma,i_1),i_2)  \\
	& \hspace{4mm}\mbox{ for } \;\;(\alpha,\beta,\gamma)\; \in \; 
	S_{N-t_2-i_2}(312)\times S_{N-t_1-i_1}(312)\times S_{\cap}(N,i_1,i_2).
	\end{align*} 
   From the above, we know that for all $\sigma \in S^{\nearrow}(N)$ there exists 
$(\tilde{\sigma},\tilde{\tau},\hat{\tau}) \in \dom_{\nearrow}$ such that
$\sigma= \ins(\tilde{\sigma},\ins(\tilde{\tau},\hat{\tau},i_1),i_2)$. Hence, $\phi_{\nearrow}$ is a surjective map. We claim that $\phi_{\nearrow}$ is 	one-to-one.
Assume that $(\alpha^1,\beta^1,\gamma^1) \in S_{N-t_2-i^1_{2}}(312)\times S_{N-t_1-i^1_{1}}(312)\times S_{\cap}(N,i^1_{1},i^1_{2})$ and 	$(\alpha^2,\beta^2,\gamma^2) \in S_{N-t_2-i^2_{2}}(312)\times S_{N-t_1-i^2_{1}}(312)\times S_{\cap}(N,i^2_{1},i^2_{2})$ and that 
$\phi_{\nearrow}((\alpha^1,\beta^1,\gamma^1))=\phi_{\nearrow}((\alpha^2,\beta^2,\gamma^2))$. 
	We have $\alpha^k \in S_{N-t_2-i^k_2}$ and 
$\ins(\beta^k,\gamma^k,i^k_1) \in S^{\Box}(t_2+i^k_2,i^k_2,j_2)$ for $k=1,2$.
	We apply Lemma \ref{lemonetoone} with $i^k=i^k_2$, 	$\tau^k=\alpha^k$ and 
$\rho^k=\ins(\beta^k,\gamma^k,i^k_1)$ and conclude that $i^1_2=i^2_2$, $\alpha^1=\alpha^2$ 
and 	$\ins(\beta^1,\gamma^1,i^1_1)=\ins(\beta^2,\gamma^2,i^2_1)$. 
Noting that $\beta^k \in S_{N-t_1-i^k_1}(312)$ and 
$\gamma^k \in S^{\Box}(t_1+t_2+i^k_1+i^k_2-N,i^k_1,j_1)$, we apply Lemma \ref{lemonetoone} 
one more time. This second time we replace 
$t$ by $(t_2+i^k_2)-(N-t_1)$ and hence conclude that  $i^1_1=i^2_1$, $\beta^1=\beta^2$ and 
$\gamma^1=\gamma^2$.   This proves that $\phi_{\nearrow}$ is a one-to-one map. Therefore 
$\phi_{\nearrow}$ is a bijection and 
 \begin{equation}
    \label{eq.SCNsumsum}
		\left|S^{\nearrow}(N)\right|\;=
		\sum_{i_1=\max\left\{1,j_1-j_2+N-t_1+1\right\}}^{j_1}\sum_{i_2=\max\left\{N-t_1+1,j_2-t_2\right\}}^{j_2} C_{N-t_1-i_1}C_{N-t_2-i_2}\left|S_{\cap}(N,i_1,i_2)\right|.
 \end{equation}
 
	For given $i_1,i_2,j_1,j_2,t_1,t_2,N$, 
the set 
$S_{\cap}(N,i_1,i_2)$ consists of all $\sigma$ in 
$S_{i_1+i_2+t_1+t_2-N}(312)$ such that 
$\sigma_{i_1}=j_1$ is a left-to-right maximum and 	$\sigma_{i_2-(N-t_1-i_1)}=j_2-(N-t_1-i_1)$ is 
a left-to-right maximum. Hence, using Krattenthaler's bijection from Section 3.1, 
these two points correspond to 
peaks at $(j_1+(i_1-1),j_1-(i_1-1))$ and $(j_2+i_2-2(N-t_1-i_1)-1,j_2-(i_2-1))$  
on the Dyck path associated with $\sigma$.  We deduce that $|S_{\cap}(N,i_1,i_2)|$ 
equals the product of the cardinalities of the three sets of Dyck segments $D_{(1)}$, $D_{(2)}$, $D_{(3)}$, where

\smallskip\noindent
$\bullet$ $D_{(1)}$ is the set of Dyck segments from $(0,0)$ to $(j_1+i_1-2,j_1-i_1)$, 

\smallskip\noindent
$\bullet$	$D_{(2)}$ is the set of Dyck segments from $(j_1+i_1,j_1-i_1)$ to $(j_2+i_2-2(N-t_1-i_1)-2,j_2-i_2)$, and

\smallskip\noindent
$\bullet$	$D_{(3)}$ is the set of Dyck segments from $(j_2+i_2-2(N-t_1-i_1),j_2-i_2)$ to $(2(i_1+i_2+t_1+t_2-N),0)$, which has the same cardinality as the set of Dyck segments from 	$(0,0)$ to $(2(i_1+i_2+t_1+t_2-N)-(j_2+i_2-2(N-t_1-i_1)),j_2-i_2)=(i_2+2t_2-j_2,j_2-i_2)$. 

\smallskip\noindent
Recalling Lemma \ref{CardDyckseg1} and Remark \ref{remDseg}, we obtain
\begin{align*}
	\left|D_{(1)}\right| \;&=\;\frac{j_1-i_1+1}{j_1}\binom{j_1+i_1-2}{j_1-1},  \\
	\left|D_{(2)}\right| \;&=\; \binom{i_1+i_2+j_2-j_1-2(N-t_1+1)}{j_2+(i_1-j_1-1)-(N-t_1)} - 
	\binom{i_1+i_2+j_2-j_1-2(N-t_1+1)}{j_2-(N-t_1)},   \\
	\left|D_{(3)}\right|\:&=\;\frac{j_2-i_2+1}{t_2+1}\binom{i_2+2t_2-j_2}{t_2}.
\end{align*}
Using the above and $\left|S_{\cap}(N,i_1,i_2) \right| = \left|D_{(1)}\right| \left|D_{(2)}\right| \left|D_{(3)}\right|  $ in Equation (\ref{eq.SCNsumsum}), the proof of part (b) is now complete.
\end{proof}     

	Notice that when $j_2=N-t_1+1$, each sum in the formula for $|S^{\nearrow}(N)|$ has only one term, namely $i_1=j_1$ in the outer sum and $i_2=j_2$ in the inner sum. 	Hence in this case
we obtain   $|D_{(1)}|\,=\,C_{j_1-1}$, $|D_{(3)}|=C_{t_2}$, and 
$\left|D_{(2)} \right|=\binom{0}{0}-\binom{0}{1}=1$, which yields the expression
\[
		\left|S^{\nearrow}(N) \right|\;=\; C_{N-t_1-j_1}\,C_{t_1-t_2-1}\,C_{j_1-1}\,C_{t_2}
		\hspace{5mm}\hbox{when $j_2=N-t_1+1$}.
\]

\section{Asymptotics of Probabilities}
   \label{sec-asymp}

The main goal of this section is to prove asymptotic formulas for the probabilities
of $S^{\bullet}(N,i,j)$, $S^{\searrow}(N)$, and $S^{\nearrow}(N)$, as described in Section  2.  The asymptotics
are based on well known asymptotics of binomial probabilities, of which the following
is a particularly useful form.

\begin{prop}\label{proplawler} (\cite{lawler}, pp.\ 61--63)  We have the relation
	\[\binom{A}{B}2^{-A}=\sqrt{\frac{2}{\pi A}}\,e^{-\frac{(2B-A)^2}{2A}}e^{O\left(\frac{1}{A}+\frac{(2B-A)^4}{A^3}\right)}\]
	and hence, if $\left|2B-A\right|\leq A^{3/4}$,
	\[\binom{A}{B}2^{-A}=\sqrt{\frac{2}{\pi A}}\,e^{-\frac{(2B-A)^2}{2A}}\left(1+O\left(\frac{1}{A}+\frac{(2B-A)^4}{A^3}\right)\right).\]
\end{prop}

The following notation will be used throughout this section.  Let 
\begin{equation}
  \label{eq-hgam}
   h(u)\;=\;\frac{\binom{2u}{u}}{2^{2u}(u+1)}\;=\;\frac{C_u}{2^{2u}}
\hspace{10mm}\hbox{and}\hspace{10mm}
\gamma(t,r)\;=\;\frac{\binom{2t-r+1}{t}}{2^{2t-r+1}}.
\end{equation}

\begin{rem}\label{remgammaandh}
The functions $h(u)$ and $\gamma(t,r)$ have the following properties.

\smallskip\noindent
(a) $\gamma(t,r)$ is decreasing in $r$ for $r \geq 1$.  This is because
\[
			\frac{\gamma(t,r+1)}{\gamma(t,r)}
			\;=\;\frac{2(t+1-r)}{2t+1-r}\;=\;\frac{2t+1-r-(r-1)}{2t+1-r}\;\leq\; 1.
\]
(b) Applying Proposition \ref{proplawler} to $h(u)$, we get $h(u)=\frac{1}{\sqrt{\pi u}(u+1)}(1+O(\frac{1}{u}))=\frac{1}{\sqrt{\pi}u^{3/2}}(1+O(\frac{1}{u}))$.

\smallskip\noindent
(c) $\frac{h(u)}{h(N)}=(\frac{N}{u})^{3/2}(1+O(\frac{1}{N}))$ if $u\asymp N$.  (This says that the
$O$ term is uniform over all $u$ and $N$ such that $cN>u>N/c$ for some fixed $c>1$.) 
\end{rem} 

Before proceeding, we
shall need the following particular form of the asymptotics of $\gamma(t,r)$.

\begin{lem}\label{gamma}
    \label{lem-gamas}
Fix $\theta\in (0,1/6)$, $\epsilon\in (0,1)$, $C>0$, and $a\in \mathbb{Z}$.  
Let $R_N\,=\,\lceil N^{1/2+\theta}\rceil$.  Then
\begin{equation}
   \label{eq.gamas1}
   \gamma(t+s,r+a) \;=\; \frac{1}{\sqrt{\pi t}}\,  
   e^{-r^2/4t}\left( 1+O\left(N^{3\theta- 1/2}\right)\right) \,,
\end{equation}
where the error term is uniform over $N$, $t$, $r$, and $s$ satisfying $|s|\leq C R_N$, 
$|r|\leq C R_N$, $\epsilon N\leq t\leq N$, and $N>N_1$ for some $N_1$.   
(Note that the error term and $N_1$
are not uniform over $\epsilon$, $a$, $C$, or $\theta$.)
\end{lem}

\begin{proof}  Using Proposition \ref{proplawler} with $A=2t+2s-r-a+1$ and $B=t+s$, we have
\begin{equation}
   \label{eq.gam1}
	\gamma(t+s,r+a)\;=\;\sqrt{\frac{2}{\pi (2t+2s-r-a+1)}}\,
	e^{-\frac{(r+a-1)^2}{2(2t+2s-r-a+1)}}\,
	\left(1+O\left(\frac{1}{N}+\frac{R_N^4}{N^3}\right)\right).
\end{equation}
Since $R_N^4/N^3 \asymp N^{4\theta -1}$ and $\theta>0$, the final term in the above 
expression is $(1+O(N^{4\theta-1}))$.
We also have
\begin{align}
   \label{eq.r24t}
	&\left|\frac{r^2}{4t} - \frac{(r+a-1)^2}{2(2t+2s-r-a+1)}\right|\;=\; O\left(\frac{R_N^3}{N^2}\right)
	\;=\; O(N^{3\theta-1/2}),  \hspace{5mm} \mbox{ and } \\
	\label{eq.root2t}
	& \sqrt{2t+2s-r-a+1}\;=\;\sqrt{2t\left(1+O\left(\frac{R_N}{N}\right)\right)} \;=\;
	\sqrt{2t}\,(1+O(N^{\theta-1/2})).
\end{align}
By Equations (\ref{eq.gam1}), (\ref{eq.r24t}) and (\ref{eq.root2t}), 
\begin{align*}
	\gamma(t+s,r+a)\;&=\;\frac{1}{\sqrt{\pi t}}\,
	e^{- \frac{r^2}{4t}}\,(1+O(N^{\theta-1/2}))(1+O(N^{3\theta-1/2}))(1+O(N^{4\theta-1})) \,.
\end{align*}
Since $0\,>\,3\theta-\frac{1}{2}\,>\,\max\{4\theta-1,\theta-\frac{1}{2}\}$ for $0<\theta < \frac{1}{6}$,
the lemma follows.
\end{proof}

We shall now prove Theorem \ref{propapprox1}, which asserts that, for fixed
 $\epsilon>0$ and $0< \theta < \frac{1}{6}$, we have
 \[
		P_N(S^{\bullet}(N,N-t,j))
		\;=\;\frac{N^{-3/2}}{2\sqrt{\pi}\left(1-\frac{N-t-j}{N}\right)^{3/2}\left(\frac{N-t-j}{N}\right)^{3/2}}\,
		\left(1+O(N^{3\theta-1/2})\right)
\]
for $\min\{j,t,N-t-j\}> \epsilon N$.

\begin{proof}[Proof of Theorem \ref{propapprox1}] 
For $j<N-t$, we know from Theorem \ref{propS_A(N,N-t,j)} that
	\begin{align*}
		P_N(S^{\bullet}(N,N-t,j))=\sum_{i_0=\max\left\{1,j-t\right\}}^{j}\frac{C_{N-t-i_0}}{C_N}\frac{(j-i_0+1)^2}{j(t+1)}\binom{i_0+2t-j}{t}\binom{i_0+j-2}{j-1}.
	\end{align*}
	Let $r=j-i_0+1$, that is, $i_0=j-r+1$. Then
	\begin{align*}
		P_N(S^{\bullet}(N,N-t,j))&=\sum_{r=1}^{\min\left\{j,t+1\right\}} \frac{r^2}{j(t+1)}\frac{C_{N-t-j+r-1}}{C_N}
		\binom{2t-r+1}{t}\binom{2j-r-1}{j-1}
		\\
		&=\frac{1}{4}\sum_{r=1}^{\min\left\{j,t+1\right\}} \frac{r^2}{j(t+1)}\frac{C_{N-t-j+r-1}}{2^{2(N-t-j+r-1)}}\frac{2^{2N}}{C_N}\frac{\binom{2t-r+1}{t}}{2^{2t-r+1}}\frac{\binom{2j-r-1}{j-1}}{2^{2j-r-1}} 
		\\
		&=\frac{1}{4}\sum_{r=1}^{\min\left\{j,t+1\right\}}\frac{r^2}{j(t+1)}\,
		\gamma(t,r)\,\gamma(j-1,r)\,\frac{h(N-t-j+r-1)}{h(N)}.
	\end{align*}

	To analyze the sum we choose the truncation point  $R_N \,=\,\lceil N^{1/2+\theta}\rceil$
and consider the sums $S'_N=\sum_{r=1}^{R_N}$ and $T_N=\sum_{r=R_N+1}^{\min\left\{j,t+1\right\}}$ separately.  Observe that $R_N\,=\,o(\min\{j,t\})$.

\smallskip
 We shall first prove that $T_N$ is very small.  
For all $r > R_N$ we have 
\begin{align*}
	\gamma(t,r) \; &\leq \; \gamma(t,R_N+1)     \hspace{8mm}
	\hbox{(by Remark \ref{remgammaandh}(a))}  \\
	& =\; \frac{\binom{2t-R_N}{t}}{2^{2t-R_N}} \\
	&= \; O\left(e^{-\frac{R_N^2}{2(2t-R_N)}}\right)   \hspace{5mm}\hbox{(by Proposition 
	\ref{proplawler} with $A=2t-R_N$ and $B=t$)}.
\end{align*} 
Since $2t-R_N \leq 2N$ and since $R_N \,=\,\lceil N^{1/2+\theta}\rceil$, we have 
\begin{equation}
  \label{eq.gamexp}
     \gamma(t,r)\;=\; O\left(e^{-(N^{1+2\theta})/4N}\right)\;=\;O\left(e^{-N^{\theta}}\right).
\end{equation}
By Remark \ref{remgammaandh}(c) and the fact that $\frac{j}{N} < \frac{N-t}{N} - \epsilon$, we obtain for $r>R_N$ that
\begin{align}\label{eq.h(N)}
		\frac{h(N-t-j+r-1)}{h(N)} \;&=\; \frac{N^{3/2}}{(N-t-j+r-1)^{3/2}}\left(1+O\left(\frac{1}{N}\right)\right)	\notag \\
		&\leq \;\left(\frac{N}{N-t-j}\right)^{3/2}\left(1+O\left(\frac{1}{N}\right)\right)  \notag \\
		&<\; \epsilon^{-3/2}\left(1+O\left(\frac{1}{N}\right)\right) \;\;\;=\;\;\;O(1).
			\end{align}
Moreover, $\gamma(j-1,r) \leq 1$ and  $\frac{r^2}{j(t+1)} \leq 1$ (since 
$r \leq \min\left\{t+1,j\right\}$ in the sum). 
Therefore
\begin{equation}
  \label{eq.TNsmall}
	T_N\;=\;\frac{1}{4}\sum_{r=R_N+1}^{\min\left\{t+1,j\right\}}\frac{r^2}{j(t+1)}\,
				\gamma(t,r)\gamma(j-1,r)\,\frac{h(N-t-j-1+r)}{h(N)}\;=\;
				O\left(Ne^{-N^{\theta}}\right).
\end{equation}

Next we approximate $S'_N.$ Using Lemma \ref{lem-gamas} and Remark 
\ref{remgammaandh}(c), we rewrite the truncated sum as
\begin{align*}
	S_N^{'}\;&=\;\frac{1}{4}\sum_{r=1}^{R_N}\frac{r^2}{j(t+1)}
				\,\gamma(t,r)\gamma(j-1,r)\,\frac{h(N-t-j+r-1)}{h(N)} \\
				&=\;\frac{1}{4 \pi}\sum_{r=1}^{R_N}\frac{r^2}{j(t+1)}\frac{e^{-\frac{r^2}{4t}}}{\sqrt{ t}}\frac{e^{-\frac{r^2}{4j}}}{\sqrt{j}}\left(\frac{N}{N-t-j+r-1}\right)^{3/2}(1+O(N^{3\theta-1/2}))^2
				\left(1+O\left(\frac{1}{N}\right)\right)   \\
				&=\;\frac{1}{4\pi}\sum_{r=1}^{R_N}\frac{r^2}{j^{3/2}t^{3/2}(\frac{N-t-j}{N})^{3/2}} e^{-\frac{r^2}{4t}}e^{-\frac{r^2}{4j}} (1+O(N^{3\theta-1/2})),
\end{align*}
where the last step used $t+1=t\,(1+O(N^{-1}))$ and 
\[
	 (N-t-j+r-1)^{3/2}\,=\;  (N-t-j)^{3/2}\left(1+\frac{r-1}{N-t-j}\right)^{3/2}\,=\;
	 (N-t-j)^{3/2}\left(1+O\left(\frac{R_N}{N}\right)\right) .
\]

Define $X_r:=\frac{r}{\sqrt{N}}$ for $r=1,2\ldots$  Then
\begin{equation}
     \label{eq.sprime}
	N^{3/2}S_N^{'}\;=\;\frac{1}{4 \pi}\sum_{r=1}^{R_N}	\; \frac{X_r^2\exp\left(\frac{-X_r^2}{4t/N}\right)\exp\left(\frac{-X_r^2}{4j/N}\right)}{\left(\frac{j}{N}\right)^{3/2}\left(\frac{t}{N}\right)^{3/2}\left(\frac{N-t-j}{N}\right)^{3/2}}\;\frac{1}{\sqrt{N}}(1+O(N^{3\theta-1/2})).
\end{equation}
Let  $K=\frac{1}{4}(\frac{N}{t}+\frac{N}{j})$.   By our assumptions, $K\asymp 1$.
By Proposition \ref{thmapprox} with $\Delta=\frac{1}{\sqrt{N}}$, we get
\begin{equation*}
	\sum_{r=1}^{R_N}X_r^2e^{-KX_r^2}N^{-1/2}
		\;=\;\int_{0}^{R_N/\sqrt{N}}z^2e^{-Kz^2}dz\;+\;O(N^{\theta-1/2}).
\end{equation*}
Hence, by Lemma \ref{remintegral}, 
\begin{align*}
	\sum_{r=1}^{R_N}X_r^2e^{-KX_r^2}N^{-1/2} \;&
		=\; \frac{\sqrt{\pi}}{4K^{3/2}}\left(1+O\left(e^{-R_N\sqrt{K/N}}\right)\right)\;+\;
		O\left(N^{\theta-1/2}\right) \\
		&=\; \frac{\sqrt{\pi}}{4K^{3/2}}\;+\;O(N^{\theta-1/2}).
\end{align*}
As we plug this into Equation (\ref{eq.sprime}) we obtain
\begin{equation*}
	S_N^{'}\;=\; N^{-3/2} \left(\frac{1}{2\sqrt{\pi}\left(1-\frac{N-t-j}{N}\right)^{3/2}\left(\frac{N-t-j}{N}\right)^{3/2}}+O(N^{\theta-1/2})\right) 	\left(1+O\left(N^{3\theta-1/2}\right)\right) \,.
\end{equation*}
Recalling Equation (\ref{eq.TNsmall}), the theorem follows.
\end{proof}

\begin{rem}\label{remrewriteapprox1}
For future reference, we note that Theorem \ref{propapprox1} and Remark \ref{remgammaandh}(b) show that for 
$0\pprec j\pprec N-t\pprec N$
we have
	\begin{align}
		\label{eq:1}
		\sum_{r=1}^{\min\left\{j,t+1 \right\}} & \frac{r^2}{j(t+1)}\,\gamma(t,r)\,\gamma(j-1,r)\,h(N-t-j+r-1)\\
\nonumber
		&=\; 4\left(\frac{h(N)N^{3/2}}{2\sqrt{\pi}(t+j)^{3/2}(N-t-j)^{3/2}}\right)\left(1+O(N^{3\theta-1/2})\right)\\
		\label{eq:3}
		&=\;\frac{2}{\pi (t+j)^{3/2}(N-t-j)^{3/2}}\left(1+O(N^{3\theta-1/2})\right)\left(1+O\left(\frac{1}{N}\right)\right)\\
		\label{eq:4}
		&=\,\sum_{r=1}^{\min\left\{j,t+1 \right\}}  \frac{r^2}{j(t+1)}\,\gamma(t,r)\,\gamma(j-1,r)\,
		\frac{(1+O(\frac{1}{N}))}{\sqrt{\pi}(N-t-j+r-1)^{3/2}}.
	\end{align}
\end{rem}

Next we shall prove Theorem \ref{propapprox2}, which says that, for fixed 
$\epsilon>0$ and $0<\theta<\frac{1}{6}$, we have
 \[
		P_N\left(S^{\searrow}(N)\right)\;=\;
	\frac{1}{4\pi}\frac{N^{-3}(1+O(N^{3\theta-1/2}))}{\left(\frac{(N-t_2-j_2)-(N-t_1-j_1)}{N}\right)^{3/2}\left(\frac{N-t_1-j_1}{N}\right)^{3/2}\left(1-\frac{N-t_2-j_2}{N}\right)^{3/2}}
\]
for $\min \{j_2,j_1-j_2,N-t_1-j_1,t_1-t_2,t_2\}> \epsilon N$.

\begin{proof}[Proof of Theorem \ref{propapprox2}]
	Let $u=j_2-i_1+1$ and $r=j_1-j_2-i_0+1$ in Theorem \ref{propS_B(N)}.  Then we obtain
{\allowdisplaybreaks	
\begin{align*}
		P_N\left(S^{\searrow}(N)\right)
		\;=\;&\sum_{u=1}^{\min\left\{j_2,t_2+1\right\}}\frac{u^2}{j_2(t_2+1)}\binom{2t_2-u+1}{t_2}\binom{2j_2-u-1}{j_2-1}	\\
		&\hspace{3mm} \sum_{r=1}^{\min\left\{j_1-j_2,t_1-t_2\right\}}\frac{r^2}{(j_1-j_2)(t_1-t_2)}\binom{2(t_1-t_2-1)-r+1}{t_1-t_2-1}\times \\
		&\hspace{3mm} \binom{2(j_1-j_2-1)-r+1}{j_1-j_2-1}\frac{C_{N-t_1-j_1+u+r-1}}{C_N}	\\
		=\;& \frac{1}{2^4}\sum_{u=1}^{\min\left\{j_2,t_2+1\right\}}\frac{u^2}{j_2(t_2+1)} \frac{\binom{2t_2-u+1}{t_2}}{2^{2t_2-u+1}} \frac{\binom{2j_2-u-1}{j_2-1}}{2^{2j_2-u-1}} \\
		&\sum_{r=1}^{\min\left\{j_1-j_2,t_1-t_2\right\}} \frac{r^2}{(j_1-j_2)(t_1-t_2)} \frac{\binom{2(t_1-t_2-1)-r+1}{t_1-t_2-1}}{2^{2(t_1-t_2-1)-r+1}} \frac{\binom{2(j_1-j_2-1)-r+1}{j_1-j_2-1}}{2^{2(j_1-j_2-1)-r+1}} \frac{2^{2N}}{C_N} \times \\
		&\hspace{3mm} \frac{C_{N-t_1-j_1+u+r-1}}{2^{2(N-t_1-j_1+u+r-1)}}   
		\\
	=\; & \frac{1}{2^4}\sum_{u=1}^{\min\left\{j_2,t_2+1\right\}}\frac{u^2}{j_2(t_2+1)}
	\, \gamma(t_2,u)\gamma(j_2-1,u) \,\times	\\
	&\sum_{r=1}^{\min\left\{j_1-j_2,t_1-t_2\right\}} 			
	\frac{r^2\gamma(t_1-t_2-1,r)\gamma(j_1-j_2-1,r) }{(j_1-j_2)(t_1-t_2)}
	\,\frac{h(N-t_1-j_1+u+r-1)}{h(N)}.
\end{align*}}

First we consider the inner sum. We substitute $t$ for $t_1-t_2-1$, $j$ for $j_1-j_2$, and $N_u$ for 
$N-t_2-j_2+u-1$, and use Remark \ref{remrewriteapprox1} [Equations (\ref{eq:1}) and (\ref{eq:3})] and Remark \ref{remgammaandh}(b) to conclude that
\begin{align*}	
	\sum_{r=1}^{\min\left\{j_1-j_2,t_1-t_2\right\}}&
		\frac{r^2\,\gamma(t_1-t_2-1,r)\gamma(j_1-j_2-1,r)}{(j_1-j_2)(t_1-t_2)}\,
		\frac{h(N-t_1-j_1+u+r-1)}{h(N)}	\\
	&=\;\sum_{r=1}^{\min\left\{j,t+1\right\}}
		\frac{r^2}{j(t+1)}\,\gamma(t,r)\gamma(j-1,r)\,\frac{h(N_u-t-j+r-1)}{h(N)} \\
	&=\frac{2(1+O(N^{3\theta-1/2}))}	{\pi(t+j)^{3/2}(N_u-t-j)^{3/2}} \sqrt{\pi}N^{3/2}\left(1+O\left(\frac{1}{N}\right)\right)    \hspace{6mm}\hbox{(since $N_u\geq \epsilon N$).}
\end{align*}
Hence, we obtain
\begin{align*}
	P_N\left(S^{\searrow}(N)\right) \;& =\;
	\frac{N^{3/2}}{8} \frac{(1+O(N^{3\theta-1/2}))}{(t+j)^{3/2}}\times \\
	&\sum_{u=1}^{\min\left\{j_2,t_2+1\right\}}
	\frac{u^2}{j_2(t_2+1)}\,\gamma(t_2,u)\,\gamma(j_2-1,u)\,
	\frac{(1+O(\frac{1}{N}))}{\sqrt{\pi}(N_u-t-j)^{3/2}}.
\end{align*}

We notice that $N_u-t-j=(N-t-j)-t_2-j_2+u-1$. 
Finally, we apply 
Remark \ref{remrewriteapprox1} [Equations (\ref{eq:3}--\ref{eq:4})] one more time, replacing 
$N$ by $N-t-j$, $t$ by $t_2$, and $j$ by $j_2$,  obtaining
\begin{align*}
	P_N \left(S^{\searrow}(N)\right)\;& =\;\frac{N^{3/2}(1+O(N^{3\theta-1/2}))}{8(t+j)^{3/2}}
	\frac{2(1+O(N^{3\theta-1/2}))}{\pi (t_2+j_2)^{3/2}((N-t-j)-t_2-j_2)^{3/2}} \\
	& =\;\frac{N^{3/2}(1+O(N^{3\theta-1/2}))}{4\pi(t_1-t_2+j_1-j_2-1)^{3/2}(t_2+j_2)^{3/2}(N-t_1-j_1+1)^{3/2}} \,.
\end{align*}
Theorem \ref{propapprox2} follows.
\end{proof} 

Next we prove Theorem \ref{propapprox3}, which says that, for fixed
$\epsilon >0$ and $0<\theta<\frac{1}{6}$, we have
	\begin{align*}
		P_N(S^{\nearrow}(N))\;=
		\frac{N^{-3}(1+O(N^{3\theta-1/2}))}{4\pi\left(1-\frac{(N-t_1-j_1)+(N-t_2-j_2)}{N}\right)^{3/2}
		\left(\frac{N-t_1-j_1}{N}\right)^{3/2}\left(\frac{N-t_2-j_2}{N}\right)^{3/2}
		}
	\end{align*}
for $\min\{j_1,N-t_1-j_1,j_2-(N-t_1), N-t_2-j_2,t_2\}> \epsilon N$.

\begin{proof}[Proof of Theorem \ref{propapprox3}] 
By Theorem \ref{propS_C(N)},
\begin{align*}
		P_N(S^{\nearrow}(N)) \;=& 	
		\sum_{i_1=\max\left\{1,j_1-j_2+N-t_1+1\right\}}^{j_1} \sum_{i_2=\max\left\{N-t_1+1,j_2-t_2\right\}}^{j_2} \frac{C_{N-t_1-i_1}C_{N-t_2-i_2}}{C_N} \times \\
		& \frac{(j_1-i_1+1)(j_2-i_2+1)}{j_1(t_2+1)} \binom{j_1+i_1-2}{j_1-1}	\binom{i_2+2t_2-j_2}{t_2} \times \\
		& \left(\binom{i_1+i_2+j_2-j_1-2(N-t_1+1)}{j_2+(i_1-j_1-1)-(N-t_1)}-\binom{i_1+i_2+j_2-j_1-2(N-t_1+1)}{j_2-(N-t_1)} \right). 
	\end{align*}

	Let $r_1=j_1-i_1+1$ and $r_2= j_2-i_2+1$.
As in the proof of Theorem \ref{propapprox2}, we use $h(u)$ and $\gamma(t,r)$ to rewrite the probability as 
 \begin{align*}
		P_N&(S^{\nearrow}(N))\; =\; \frac{1}{2^4} \,\times \\
	&\sum_{r_1=1}^{\min\left\{j_1,j_2-N+t_1\right\}}\sum_{r_2=1}^{\min\left\{t_2+1,j_2-N+t_1 \right\}} \, \frac{r_1}{j_1} \frac{r_2}{t_2+1} h(N-t_1-j_1+r_1-1) 	\gamma(j_1-1,r_1) \gamma(t_2,r_2) \, \times
	\\
	& \frac{h(N-t_2-j_2+r_2-1)}{h(N)} \left( \gamma(j_2+t_1-N-r_1,r_2-r_1+1) - \gamma(j_2+t_1-N,r_1+r_2+1) \right).
\end{align*}
To analyze the sum we choose the truncation point $R_N =\lceil N^{1/2+\theta}\rceil$  and consider the sums $S'_N=\sum_{r_1=1}^{R_N}\sum_{r_2=1}^{R_N}$ and  
$T_N=P_N(S^{\nearrow}(N))-S^{'}_N$.

We  first show that $T_N$ is very small.  We view $T_N$ as a sum over pairs $(r_1,r_2)$
in which at least one of $r_1$ or $r_2$ is greater than $R_N$. If $r_1 > R_N$, then 
$\gamma(j_1-1,r_1)\,=\,O(e^{-N^{\theta}})$ (recalling the argument for Equation (\ref{eq.gamexp})). 
Similarly, if $r_2 > R_N$, then $\gamma(t_2,r_2)\,=\,O(e^{-N^{\theta}})$. 
As in Equation (\ref{eq.h(N)}), we know that 
$\displaystyle \frac{h(N-t_2-j_2+r_2-1)}{h(N)}\,=\,O(1)$.
Also, $ h(N-t_1-j_1+r_1-1)\leq 1$, $\gamma (t_2,r_2) \leq 1$, and $\gamma(j_1-1,r_1) \leq 1$.
Moreover, for $r_1 \leq \min\left\{j_1,j_2-N+t_1\right\}$ and $r_2 \leq \min\left\{t_2+1,j_2-N+t_1 \right\} $ we get $\displaystyle \frac{r_1}{j_1}\frac{r_2}{t_2+1} \leq 1$.  Thus the largest term in 
$T_N$ is 
$O(e^{-N^{\theta}})$, and hence $T_N\,=\,O(N^2)\,O(e^{-N^{\theta}})\,=\,O(e^{-N^{\theta}/2})$.

Next, we approximate $S_N^{'}$.
For the rest of the proof, we will write $-\xi=3\theta-1/2$.
By Lemma \ref{lem-gamas} we have that
$\displaystyle \gamma(t_2,r_2)=\frac{1}{\sqrt{\pi t_2}}e^{- \frac{r_2^2}{4t_2}}(1+O(N^{-\xi}))$, 
$ \displaystyle \gamma(j_1-1,r_1)=\frac{1}{\sqrt{\pi j_1}}e^{- \frac{r_1^2}{4j_1}}(1+O(N^{-\xi}))$
and 
\begin{align*} 
\gamma(&j_2+  t_1-N-r_1,r_2-r_1+1)- \gamma(j_2+t_1-N,r_1+r_2+1) \\
&=\;\frac{1}{\sqrt{\pi}\sqrt{(j_2+t_1-N)}}\left(e^{-\frac{(r_2-r_1)^2}{4(j_2+t_1-N)}}-e^{-\frac{(r_1+r_2)^2}{4(j_2+t_1-N)}}+O(N^{-\xi})\right)   \,.
\end{align*}
(For the difference, we need to be careful about relative errors:  we use 
$A_N(1+O(N^{-\xi}))-B_N(1+O(N^{-\xi})) \,=\, A_N-B_N+\max\{A_N,B_N\}O(N^{-\xi})$.)
We also use Remark \ref{remgammaandh}(b,c), 
and rewrite $S_N^{'}$ as 
\allowdisplaybreaks {  
\begin{align*}
S_N^{'}\;=&\;\frac{1}{2^4}\sum_{r_1=1}^{R_N}\sum_{r_2=1}^{R_N} \frac{r_1}{j_1} \frac{r_2}{t_2+1} 
\frac{e^{- r_1^2/4j_1}\, e^{- r_2^2/4t_2}}{\sqrt{\pi j_1} \sqrt{\pi t_2} \sqrt{\pi}(N-t_1-j_1+r_1-1)^{3/2}} \times
\\
&\hspace{5mm} \frac{1}{\sqrt{\pi (j_2+t_1-N)}}\left(e^{-\frac{(r_2-r_1)^2}{4(j_2+t_1-N)}}
- e^{-\frac{(r_1+r_2)^2}{4(j_2+t_1-N)}} + O(N^{-\xi}) \right) \times \\
& \hspace{5mm}
  \left(\frac{N}{N-t_2-j_2+r_2-1}\right)^{3/2}\left(1+O\left(N^{-1}\right)\right)(1+O(N^{-\xi}))
\\
=&\;\frac{1}{2^4\pi^{2}}\sum_{r_1=1}^{R_N}\sum_{r_2=1}^{R_N} \frac{r_1}{j_1} \frac{r_2}{t_2+1} 
\frac{e^{-r_1^2/4j_1}\, e^{- r_2^2/4t_2} \,(1+O(N^{-\xi}))}
{\sqrt{j_1} \sqrt{t_2} (N-t_1-j_1)^{3/2}(\frac{N-t_2-j_2}{N})^{3/2}} \times\\
&\hspace{5mm} \frac{1}{\sqrt{(j_2+t_1-N)}}\left(e^{-\frac{(r_2-r_1)^2}{4(j_2+t_1-N)}}-e^{-\frac{(r_1+r_2)^2}{4(j_2+t_1-N)}}+ O(N^{-\xi})\right).
\end{align*}}

Define $X_{r_1}=\frac{r_1}{\sqrt{N}}$ and $X_{r_2}=\frac{r_2}{\sqrt{N}}$.  Then
\begin{align*}
	&N^3S_N^{'} = \frac{1}{2^4\pi^{2}} \sum_{r_1=1}^{R_N}\sum_{r_2=1}^{R_N} 	\frac{X_{r_1}X_{r_2} \exp(\frac{-X_{r_1}^2}{4j_1/N}) \exp(\frac{-X_{r_2}^2}{4t_2/N})  (1+O(N^{-\xi}))} 	{\left(\frac{j_1}{N}\right)^{3/2}\left(\frac{t_2}{N}\right)^{3/2}\left(\frac{N-t_1-j_1}{N}\right)^{3/2}\left(\frac{N-t_2-j_2}{N}\right)^{3/2}\left(\frac{j_2+t_1-N}{N}\right)^{1/2}} \times \\
	& \left(\exp\left(\frac{-(X_{r_2}-X_{r_1})^2}{\frac{4(j_2+t_1-N)}{N}}\right)-\exp\left(\frac{-(X_{r_1}+X_{r_2})^2}{\frac{4(j_2+t_1-N)}{N}}\right)+ O(N^{-\xi})\right) \frac{1}{\sqrt{N}} \frac{1}{\sqrt{N}}.
\end{align*}

Let $K_1=\frac{N}{4t_2}$, $K_2=\frac{N}{4j_1}$ and $K_3=\frac{N}{4(j_2+t_1-N)}$. Hence, by 
Proposition \ref{thmapprox2}, 
\begin{align*}
	& N^3 S_N^{'} \; = 
	 \\ &
	\frac{\int_{0}^{R_N/\sqrt{N}}\int_{0}^{R_N/\sqrt{N}} xye^{-K_1x^2}e^{-K_2y^2}
	(e^{-K_3(y-x)^2}-e^{-K_3(x+y)^2}+O(N^{-\xi}))dx\,dy + O(N^{-\xi-\theta})} {
	2^4\,\pi^{2}\,\left(\frac{j_1}{N}\right)^{3/2}\left(\frac{t_2}{N}\right)^{3/2}
	\left(\frac{N-t_1-j_1}{N}\right)^{3/2}\left(\frac{N-t_2-j_2}{N}\right)^{3/2}
	\left(\frac{j_2+t_1-N}{N}\right)^{1/2}}   \\
	& \hspace{11mm} \times \left(1+O(N^{-\xi})\right) \,.
\end{align*}
Hence, we apply Lemma \ref{doubleintegral} with $w_1=w_2=\frac{R_N}{\sqrt{N}}$ and obtain that 
\begin{align*}
	P_N(S^{\nearrow}(N)) 
	\;=\; \frac{N^{-3}(1+O(N^{-\xi}))}{4\pi\left(1-\frac{(N-t_1-j_1)+(N-t_2-j_2)}{N}\right)^{3/2}\left(\frac{N-t_1-j_1}{N}\right)^{3/2}\left(\frac{N-t_2-j_2}{N}\right)^{3/2}}.
\end{align*} 
\end{proof}

\section{The Lower Right Corner}
    \label{sec-RW}

We begin with the proof of Proposition \ref{k-1-case}.
Recall the function $\rho$ from Definition \ref{def-rho}.

\begin{proof}[Proof of Proposition  \ref{k-1-case}](a)  Fix $a,b\in \mathbb{N}$.  We shall 
prove that $P_N(S^{\bullet}(N,N{-}a{+}1,b) ) = \rho(a,b)(1+O(N^{-1})).$ 
For  $ N> a+b$, we have from Theorem \ref{propS_A(N,N-t,j)} that
\begin{equation*}
\frac{\left|S^{\bullet}(N,N{-}a{+}1,b)\right|}{C_N} 
\;= \sum_{i_0=\max \left\{ 1,b-a+1 \right\}}^{b} \frac{C_{N-a-i_0+1}}{C_N} \frac{(b-i_0+1)^2}{ba}\binom{i_0+2(a{-}1)-b}{a-1}\binom{i_0{+}b{-}2}{b-1} \,.
\end{equation*}
By the fact that $\frac{C_k}{C_{k+1}}= \frac{1}{4}(1+\frac{3}{2k+1})$ and since $i_0$ is bounded, we get that
\begin{align*}
P_N(S^{\bullet}(N,&N{-}a{+}1 ,b)) \\  & = 
\sum_{i_0=\max \left\{ 1,b-a+1 \right\}}^{b} \frac{(1+O(N^{-1}))}{4^{i_0+a-1}}\frac{(b-i_0+1)^2}{ba}\binom{i_0+2(a{-}1)-b}{a-1}\binom{i_0+b-2}{b-1} \\
& =\; \rho(a,b)\,(1+O(N^{-1}))\,.
\end{align*}

\noindent
(b) The proof is by induction on $k$. The $k=1$ case is part (a). Assume that the result holds for $k-1$.
Assume $N$ is large enough that $B_m <  N-A_m+1 $ for all $m = 1, \ldots, k$. 
By Theorem \ref{propS_B(N)}, for all $k \geq 2$ we have that
\begin{align*}
& \frac{\left|S^{\searrow k} 
(N,N{-}A_k{+}1,\ldots, N{-}A_1{+}1,B_k, \ldots,B_1)  \right|}{C_N} \\
& =\sum_{i = \max \left\{1,B_1{-}A_1{+}1 \right\}}^{B_1} \frac{\left|S^{\Box}(i+A_1-1,i,B_1)\right|}{C_N} \left|S^{\searrow k-1}(\tilde{N},\tilde{N}{-}\tilde{A}_k{+}1,\ldots,\tilde{N}{-}\tilde{A}_2{+}1,\tilde{B}_k,\ldots,\tilde{B}_2)  \right|
\end{align*}
where $\tilde{N} \equiv \tilde{N}(i) = N{-}A_1{-}i{+}1$, $\tilde{A}_m= A_m - A_1 $, 
$\tilde{B}_m = B_m-B_1 $ for all $2 \leq m \leq k$.  Then
\begin{align*}
 P_N&(S^{\searrow k}(N,N-A_k+1,\ldots,N-A_1+1,B_k,\ldots,B_1)) \\ 
& = \sum_{i = \max \left\{1,B_1-A_1+1 \right\}}^{B_1}\left|S^{\Box}(i+A_1-1,i,B_1)\right| \frac{C_{\tilde{N}}}{C_N} \times \\
& \hspace{15mm} \frac{\left|S^{\searrow k-1}(\tilde{N},\tilde{N}{-}\tilde{A}_k{+}1,\ldots,
  \tilde{N}{-}\tilde{A}_2{+}1,\tilde{B}_k,\ldots,\tilde{B}_2)  \right|}{C_{\tilde{N}}}\\
& = \sum_{i = \max \left\{1,B_1-A_1+1 \right\}}^{B_1} \frac{\left|S^{\Box}(i+A_1-1,i,B_1)\right|}{4^{i+A_1-1}}(1+O(N^{-1}))
 P_{\tilde{N}}(\sigma_{\tilde{N}-\tilde{A}_2+1}=\tilde{B}_2, \ldots,\sigma_{\tilde{N}-\tilde{A}_{k}+1}=\tilde{B}_{k}) \\
& = \sum_{i = \max \left\{1,B_1-A_1+1 \right\}}^{B_1} \frac{\left|S^{\Box}(i+A_1-1,i,B_1)\right|}{4^{i+A_1-1}}(1+O(N^{-1}))
\times \prod_{v=2}^{k} \rho(a_v,b_v) \\ 
& = \left[ \prod_{v=1}^{k} \rho(a_v,b_v) \right] (1+O(N^{-1})) \,.
\end{align*}
The last two equations follow by the inductive step assumption. 
\end{proof}

Our next task is to prove Theorem \ref{thm-limX}, which says that the random variables 
$\{X_q^N\}$ have a limit $\{X_q\}$ as $N\rightarrow\infty$.  This is facilitated with 
some notation.

Let $u=(u_1,u_2), v=(v_1,v_2) \in \mathbb{Z}^2$. We write $u \nwarrow v $ if 
$u$ is ``northwest'' of $v$, i.e.\ if $u_1 < v_1$ and $u_2 > v_2$. 
Thus $Q = \left\{v \in \mathbb{Z}^2: v \nwarrow 0 \right\}$. Let Seq$\nwarrow$ be the set of all finite and infinite subsets $T= \left\{t^{(1)}, t^{(2)}, t^{(3)}, \ldots \right\}$ of $Q$ with 
$t^{(i+1)}\nwarrow t^{(i)}$ for each $i$. 

For finite subsets $D$ and $F$ of $Q$, let 
\[    \theta_N(D,F) \,=\,P_N(X^N_q=1 \; \forall q\in D, \textrm{ and }
              X^N_r=0 \;\forall r\in F ) \hspace{5mm}\hbox{for $N\in \mathbb{N}$},
\]
and let $\theta(D,F)\,=\lim_{N\rightarrow\infty}\theta_N(D,F)$ if this limit exists.

\begin{proof}[Proof of Theorem \ref{thm-limX}]
By Kolmogorov's Extension Theorem, it suffices to prove that the limit 
$\theta(D,F)$ exists for all disjoint finite subsets $D$ and $F$ of $Q$.
Proposition \ref{k-1-case} shows that $\theta(D,\emptyset)$ exists whenever $D$ is
a finite subset of Seq$\nwarrow$.  If $-i_1<-i_2$ and $j_1<j_2$, then Proposition 
\ref{propS_C(N)}(a) shows that $\theta_N(\{(-i_1,j_1),(-i_2,j_2)\},\emptyset) \,=\,0$ for
sufficiently large $N$.  It follows that $\theta(D,\emptyset)$ exists and equals 0 whenever
$D$ is a finite subset of $Q$ that is \underline{not} in Seq$\nwarrow$.

Let $D$ and $F$ be disjoint finite subsets of $Q$.  The following argument is a
generalization of the proof in Kingman \cite{Kingman} for one-dimensional regenerative 
sequences.  We have
\begin{align}
   \nonumber
   \theta_N(D,F)\;&=\;  E_N\left(  \left(\prod_{q\in D}X^N_q \right)\,\prod_{r\in F}(1-X^N_r)\right)   \\
      \nonumber
       & = \;   E_N\left(  \left(\prod_{q\in D}X^N_q \right)\,\sum_{G\subset F} \prod_{s\in G}(-X^N_s)
           \right)   \\
           \nonumber
       & = \;  \sum_{G\subset F} (-1)^{|G|}E_N\left(\prod_{s\in D\cup G}X^N_s \right)  \\
       \label{eq-incexc}
       & = \;  \sum_{G\subset F} (-1)^{|G|} \theta_N(D\cup G,\emptyset).
\end{align}
  From the previous paragraph,
we know that the final expression converges as $N\rightarrow\infty$.  Hence $\theta_N(D,F)$
converges. 
\end{proof}

Having proven the above theorem, we know that 
\[    \theta(D,F) \,=\,P_{\infty}(X_q=1 \; \forall q\in D, \textrm{ and }
              X_r=0 \;\forall r\in F ) \hspace{5mm}\hbox{for finite $D,F\subset Q$}.
\]

Some properties of the limiting collection $\{X_{q}\}$ follow immediately from 
Proposition \ref{k-1-case}.    In particular, $P_{\infty}(X_{(-i,j)}=1)=\rho(i,j)=1-P_{\infty}(X_{(-i,j)}=0)$.
More generally, 
we see from Proposition \ref{k-1-case}(b) that if $T$ is a finite subset of $Q$ of the form
\begin{align}
   \nonumber
  T\,=\,& \{ (-A_m,B_m) \,:\,m=1,\ldots,k\}  \hspace{3mm}\hbox{where}  \\
  \label{T-decr}
  &A_m \,=\, \sum_{l=1}^m a_l, \;\, B_m \,=\, \sum_{l=1}^m b_l \hbox{ for }m=1,\ldots,k  
   \hbox{ and }a_1,\ldots,a_k,b_1,\ldots,b_k \in \mathbb{N}
\end{align}
(in particular  $T\in$ Seq$\nwarrow$), 
then $\theta(T,\emptyset) \,=\, \prod_{m=1}^k\rho(a_m,b_m)$.  As we mentioned in section
\ref{sec-results},
this is a kind of two-dimensional regenerative property that corresponds to the 
fact (Theorem \ref{thm-RW}, proven below) that the random set of 
points $q$ where $X_q$ is $1$ follows the law of a trajectory of a random walk that can 
only move up and left---i.e., a two-dimensional analogue of a renewal process.
Another special case we consider is $\theta(T,B(T))$ where $B(T)= \left[-A_k,-1\right] \times \left[1,B_k\right] \setminus T$; this is the probability that $T$ is exactly the set
of locations of all the 1's in the smallest rectangle containing $T$ and the bottom right corner
of $Q$.
Proposition \ref{prop64} below shows that $\theta(T,B(T))$ also has a product form.
\begin{notation}
For $T$ of the form (\ref{T-decr}) let
\begin{align*}
\Lambda_N(T)= & \left\{ \sigma \in S_N(312): \sigma_{N-A_m+1}=B_m  \hspace{3mm}\hbox{for}\hspace{3mm}m=1,\ldots,k \hspace{3mm} \mbox{and}\hspace{3mm} \sigma_{N-t+1} >B_k \hspace{3mm} \mbox{for} \right. \\ 
& \left . t \in \left[1,A_k\right] \setminus   \left\{A_1,\ldots,A_k\right\} \right\}. 
\end{align*}
\end{notation}

\begin{thm}\label{theorem63} 
Let $T$ be of the form (\ref{T-decr}). Then for $N\geq A_k+B_k$,
\begin{align*}
P_N(\Lambda_N(T))\;=\;\frac{C_{N-A_k-B_k+k}}{C_N}\prod_{m=1}^{k}C_{a_m-1}C_{b_m-1}.
\end{align*}
\end{thm}

\begin{proof}
We will give a proof by induction on $k$. First we shall prove the $k=1$ case, i.e. 
\[
  \left| \Lambda_N(\left\{\left(-A_1,B_1\right)\right\}) \right| 
   \;=\; \displaystyle C_{N-A_1-B_1+1} \,C_{a_1-1}\,C_{b_1 -1}\,.
\]
Consider the map $\phi_{\bullet;N,t,j}$ in Definition \ref{phiA}. We take $t= A_1{-}1$ and $j= B_1$. 
Let $\mathbb{I}$ be the image of $S_{N-(A_1-1)-B_1}(312) \times S^{\Box}(A_1+B_1-1,B_1,B_1)$ under $\phi_{\bullet;N,t,j}$.  
We will show that 
$\mathbb{I}= \Lambda_N(\left\{\left(-A_1,B_1\right)\right\})$.   

Let $\tau \in S_{N-A_1-B_1+1}(312)$ and $\rho \in S^{\Box}(A_1+B_1-1,B_1,B_1)$, 
and let $\psi = \phi_{\bullet;N,(A_1-1),B_1}((\tau,\rho))= \ins(\tau,\rho,B_1)$.
Then $\psi_{N-A_1+1}=B_1$ by Proposition \ref{propS_A(N+M,i+M,j)}.
 By the Pigeonhole Principle, $\rho_m > B_1$ for all $m > B_1$.  Therefore,
 by Definition \ref{defiinsert}, for $u\in [1,A_1)$ we have 
 \[   \psi_{N-u+1} \;\geq \; \rho_{N-u+1-(N-A_1-B_1+1)} \;=\;\rho_{B_1+A_1-u} \;>\,B_1  \]
 (using the above observation, since $B_1{+}A_1{-}u>B_1$).
 Hence $\psi \in \Lambda_N(\left\{\left(-A_1,B_1\right)\right\}$, which proves that 
$\mathbb{I} \subseteq \Lambda_N(\left\{\left(-A_1,B_1\right)\right\})$.

Now let $\sigma \in \Lambda_N(\left\{\left(-A_1,B_1\right)\right\}) $ and let 
$i_0 = \min \left\{i \in \left[1,N \right] : \sigma_i > B_1 \right\}$. 
By Lemma \ref{lemmain}(a,b), $i_0 \in \left[1,N-A_1\right]$ and $\sigma_i > B_1$ for all 
$i \in \left[i_0,N-A_1\right].$ Since also $\sigma_i > B_1$ for all 
$i \in \left(N-A_1+1,N\right]$, we see that $\sigma_i < B_1$ if and only if $i < i_0$; hence $i_0 = B_1$.
By Lemma \ref{lemmain}(f), we can write $\sigma = \ins(\tilde{\sigma},\hat{\sigma},B_1)$ where 
$\hat{\sigma} \in S^{\Box}(A_1+B_1-1,B_1,B_1)$ and $\tilde{\sigma} \in S_{N-A_1-B_1+1}(312)$.
Hence $\sigma \in \mathbb{I}$ and we conclude that $\mathbb{I}= \Lambda_N(\left\{\left(-A_1,B_1\right)\right\})$.  
By Lemma \ref{lemonetoone} the restricted map is one-to-one and hence a bijection with $\Lambda_N(\left\{\left(-A_1,B_1\right)\right\})$. We conclude that
\begin{align*}
\left| \Lambda_N(\left\{\left(-A_1,B_1\right)\right\} \right|
 &\;=\; \left|S_{N-(A_1-1)-B_1}(312)\right| \cdot \left|S^{\Box}(A_1+B_1-1,B_1,B_1)\right| \\
 &\;=\; C_{N-A_1-B_1+1} \,C_{B_1-1}\,C_{A_1-1}= C_{N-A_1-B_1+1}\,C_{b_1-1}\,C_{a_1-1},
\end{align*} 
where we have used Remark \ref{remMadras}. This proves the result for $k=1$.

For the induction step, we assume that the statement is true for $k-1$.
Given $T=\left\{(-A_m,B_m): m=1,\ldots,k\right\} \in \textrm{Seq}{\nwarrow}$, let 
$T^*=\left\{(-A_m,B_m)-(-A_1,B_1): m=2,\ldots,k\right\} = \left\{(-A_m^*,B_m^*): m=2,\ldots,k \right\} \in \textrm{Seq}{\nwarrow}$ where $A_m^*=\sum_{l=2}^{m}a_{l}=A_m{-}A_1$ and 
$B_m^*=\sum_{l=2}^{m}b_{l}=B_m{-}B_1$. 
Let $\hat{\phi}$ be the restriction of the map $\phi_{A;N,A_1-1,B_1}$  to the domain 
$\Lambda_{N-A_1-B_1+1}(T^*) \times S^{\Box}(A_1+B_1-1,B_1,B_1)$, and let $\hat{\mathbb{I}}$
be the image of $\hat{\phi}$.  We claim that $\hat{\mathbb{I}}\,=\,\Lambda_N(T)$ and that 
$\hat{\phi}$ is a bijection. Let $\tau \in \Lambda_{N-A_1-B_1+1}(T^*)$ and $\rho \in S^{\Box}(A_1+B_1-1,B_1,B_1)$, and let $\psi\,=\,\hat{\phi}(\tau,\rho)\,=\,\ins(\tau,\rho,B_1)$. 
As in the $k{=}1$ case, we have $\psi_{N-A_1+1}=B_1$.  For $m\in [2,k]$, we have
$B_1\leq N{-}A_m{+}1<N{-}A_1{+}1$, so 
\[  \psi_{N-A_m+1} \;=\;  \tau_{(N-A_m+1)-B_1+1}+\rho_{B_1}  \;=\;
   \tau_{(N-A_1-B_1+1)-A_m^*+1}+B_1 \;=\; B_m^*+B_1 \;=\; B_m \,.
\]
Now it is not hard to see that $\psi\in  \Lambda_N(T)$.   This shows that $\hat{\mathbb{I}}
\subset \Lambda_N(T)$.

Next we show that  $\Lambda_N(T)\subset \hat{\mathbb{I}}$. 
Let $\sigma \in \Lambda_N(T)$. Let $i_0 = \min \left\{i \in \left[1,N \right] : \sigma_i > B_1 \right\}$. 
Then $i_0=B_1$ (as shown in $k=1$ case), and, by Lemma \ref{lemmain}(f),  
$\sigma = \ins(\tilde{\sigma},\hat{\sigma},B_1)$ where 
$\hat{\sigma}=\patt(\sigma_1,\ldots,\sigma_{B_1-1},\sigma_{N-A_1+1},\ldots,\sigma_{N}) \in S^{\Box}(A_1+B_1-1,B_1,B_1)$ and $\tilde{\sigma} = 
\patt(\sigma_{B_1},\ldots,\sigma_{N-A_1}) \in S_{N-A_1-B_1+1}(312)$.
Moreover, for $2\leq m\leq k$,
\[  \tilde{\sigma}_{(N-A_1-B_1+1)-A_m^{*}+1} = \tilde{\sigma}_{(N-A_m+1)-B_1+1} = \sigma_{N-A_m+1} - B_1 = B_m - B_1 = B_m^{*}.
\]
From here it is not hard to show that $\tilde{\sigma} \in \Lambda_{N-A_1-B_1+1}(T^*)$ and 
hence $\psi\in \hat{\mathbb{I}}$.  This proves $\Lambda_N(T)\subset \hat{\mathbb{I}}$.
Again, by Lemma \ref{lemonetoone} the map $\hat{\phi}$ is one-to-one. This verifies the claimed
bijection,  and we  conclude that
\begin{align*}
\left| \Lambda_N(T) \right|=\left| \Lambda_{N-A_1-B_1+1}(T^*) \right|\left|S^{\Box}(A_1+B_1-1,B_1,B_1)\right|.
\end{align*}
By the inductive step and Remark \ref{remMadras}, the result follows. 
\end{proof}

\medskip

The following result is a direct corollary.  Recall that for $T$ of the form (\ref{T-decr}), we
define $B(T)=[-A_k,-1]\times [1,B_k]\setminus T$. Also recall the function $\pi(-a,b)$ 
from Definition \ref{def-pi}.

\begin{prop}\label{prop64}
Let $T$ be of the form (\ref{T-decr}). 
Then 
\begin{align*}
 \theta(T,B(T)) \;=\; \lim_{N \rightarrow \infty} P_N(\Lambda_N(T)) 
 \;=\; \prod_{m=1}^{k} \frac{C_{b_m-1}C_{a_m-1}}{4^{b_m+a_m-1}} \;=\;
  \prod_{m=1}^k \pi(-a_m,b_m) \,.
\end{align*}
\end{prop}
\begin{proof}  This follows from Theorem  \ref{theorem63} and Remark \ref{remgammaandh}(c)
together with the observation that $\theta_N(T,B(T)) \,=\,P_N(\Lambda_N(T))$.
\end{proof}

Next we give the proof of Theorem \ref{thm-RW} which states that the set
$W^{*}=\left\{q \in Q : X_q=1\right\}$ is an infinite random member of Seq${\nwarrow}$
of the form $\left\{\vec{V}_m: m\in\mathbb{N}\right\}$ 
where $\left\{(\vec{V}_m-\vec{V}_{m-1}): m \in \mathbb{N}\right\}$ are i.i.d.\ $Q$-valued random 
vectors with distribution $\pi(-a,b)$.   Recall that for finite $D,F\subset Q$, we have
$\theta(D,F)\,=\,P_{\infty}(D\subset W^{*}, F\cap W^{*}=\emptyset)$.
\begin{proof}[Proof of Theorem \ref{thm-RW}]
As shown in the proof of Theorem \ref{thm-limX}, $\theta(D,\emptyset)=0$ whenever 
$D$ is a finite subset of $Q$ that is {\em not} in Seq${\nwarrow}$.  Therefore,
with probability one, $W^{*}$ is a (finite or infinite) member of Seq${\nwarrow}$.
Write the elements of $W^*$ as $\vec{V}_1,\vec{V}_2,\ldots$, where 
$\vec{V}_{m}\nwarrow \vec{V}_{m-1}$ for every $m=1,\ldots,|W^*|$ (where $\vec{V}_0=(0,0)$).

For $k\in \mathbb{N}$, consider $T$ of the form (\ref{T-decr}).  Then
\[   \theta(T,B(T)) \;=\;  P_{\infty}(|W^*|\geq k \hbox{ and } \vec{V}_m=(-A_m,B_m)
    \hbox{ for }m=1,\ldots,k ) \,.
 \]
Hence, by Proposition \ref{prop64} and the fact that 
$\pi$ is a probability distribution on $Q$, we see that  
\[  P_{\infty}\left(\left|W^*\right| \geq k \right) \;= {\displaystyle \scriptsize
            \sum_{
            \begin{array}{c}
             T \in Seq\nwarrow \\
             \left|T\right|=k
             \end{array}}} \theta(T,B(T)) \;=\; \left( \sum_{(-a,b)\in Q}\pi(-a,b)\right)^k \;=\; 1.
\]
Since $k$ is arbitrary,  $W^{*}$ must be infinite with probability 1.  
The above product form of $\theta(T,B(T))$ shows that
the jumps $\{\vec{V}_m-\vec{V}_{m-1}\}$ are i.i.d.\ with common distribution $\pi$.

Finally, each component of $\vec{V}_1$ has infinite mean since 
$\frac{C_{a-1}}{4^a}= \frac{1}{a\,4^a}  \binom{2a-2}{a-1}  \;\asymp\;  a^{-3/2}.$
\end{proof}

We now shift our focus from the points northwest of the origin to the points northwest
of a given point $(N-t,j)$ below the diagonal, and show that, conditional on
$X^N_{(-t,j)}=1$, we get the same limiting probabilities of nearby configurations as we do
in the bottom right corner.

\begin{prop}\label{prop-conditional}
For $m=1,\ldots,u$, let $a_m,b_m \in \mathbb{N}$ and let
$A_m= \sum_{l=1}^{m} a_l$, $B_m = \sum_{l=1}^{m} b_l$.  Then
\[
\lim_{N-t-j \rightarrow \infty} 
  P_N(\sigma_{N-t-A_{1}+1} =j{+}B_{1},  \ldots, \sigma_{N-t-A_{u}+1} = j{+}B_{u} \, | \, 
   \sigma_{N-t+1} = j) 
= \prod_{v=1}^u \rho(a_v,b_v).
\]
\end{prop}
\begin{proof}
Let $k=u+1$.  By  Theorem \ref{propS_B(N)},
\begin{align*}
& \frac{\left|S^{\searrow k}(N, N{-}t{-}A_{k-1}{+}1, \ldots, N{-}t{+}1, j{+}B_{k-1}, \ldots, j)\right|}{\left|S^{\bullet}(N,N{-}t{+}1,j)\right|} \\
& \hspace{15mm} = \;\frac{1}{\left|S^{\bullet}(N,N{-}t{+}1,j)\right|}\sum_{i= \max \{1,j-t+1\} }^{j} 
  \left|S^{\Box}(i+t-1,i,j)\right| C_{\tilde{N}}  \times \\
& \hspace{25mm}\frac{\left|S^{\searrow{k-1}}(\tilde{N},\tilde{N}{-}A_{k-1}{+}1, \ldots, 
  \tilde{N}{-}A_1{+}1,B_{k-1},\ldots,B_{1})\right|}{C_{\tilde{N}}}
\end{align*}
where $\tilde{N} \equiv \tilde{N}(i) = N{-}t{-}i{+}1$. By Proposition \ref{k-1-case}(b) and since $ \tilde{N} \geq N{-}t{-}j{+}1$, we know that
\begin{align*}
& \frac{\left|S^{\searrow k-1}(\tilde{N},\tilde{N}{-}A_{k-1}{+}1,\ldots, \tilde{N}{-}A_{1}{+}1, B_{k-1},\ldots,B_{1})\right|}{C_{\tilde{N}}}  \;=\;
 \left[ \prod_{v=1}^{k-1} \rho(a_v,b_v) \right]  \left(1+O\left(\frac{1}{N{-}t{-}j}\right)\right).
\end{align*}
Using this and Theorem \ref{propS_A(N,N-t,j)} gives
\begin{align*}
& \frac{\left|S^{\searrow k}(N, N{-}t{-}A_{k-1}{+}1, \ldots, N{-}t{+}1, j{+}B_{k-1}, \ldots, j)\right|}{\left|S^{\bullet}(N,N-t+1,j)\right|}  \\
& \hspace{15mm}
 = \;
\sum_{i= \max \left\{1,j-t+1 \right\} }^{j} \left|S^{\Box}(i+t-1,i,j)\right| C_{\tilde{N}} 
\, \frac{\left[ \prod_{v=1}^{k-1}\rho(a_v,b_v)\right] (1+O(\frac{1}{N-t-j}))}{\left|S^{\bullet}(N,N-t+1,j)\right|} \\
&\hspace{15mm}
 = \; \left[ \prod_{v=1}^{k-1}\rho(a_v,b_v) \right] \left(1+O\left(\frac{1}{N{-}t{-}j}\right)\right) \,.
\end{align*}
The result follows from the last equality. 
\end{proof}

Our final result, Theorem \ref{thrm-conditional}, follows from Proposition \ref{prop-conditional}
and an argument very similar to the proof of Theorem  \ref{thm-limX}.

\section{Acknowledgements}
We thank Alexey Kuznetsov and Sam Miner for helpful discussions.  The research of NM was supported in part by a Discovery Grant from NSERC of Canada.

\bibliography{pattern}

\begin{thebibliography}{10}

\bibitem{madras}
M.~Atapour and N.~Madras.
\newblock Large deviations and ratio limit theorems for pattern-avoiding
  permutations.
\newblock {\em Combinatorics, Probability and Computing}, 23:160--200, 2014.

\bibitem{bandlow}
J.~Bandlow and K.~Killpatrick.
\newblock An area-to-inv bijection between {D}yck paths and 312-avoiding
  permutations.
\newblock {\em Electronic Journal of Combinatorics}, 8 (2):Research Paper 40,
  2001.

\bibitem{bona2}
M.~B{\'o}na.
\newblock Exact enumeration of 1342-avoiding permutations a close link with
  labeled trees and planar maps.
\newblock {\em J. Combin. Theory Ser. A}, 80(2):257--272, 1997.

\bibitem{bona1}
M.~B{\'o}na.
\newblock {\em Combinatorics of Permutations}.
\newblock Chapman and Hall/CRC, Boca Raton, Florida, 2004.

\bibitem{bouvel}
M.~Bouvel and D.~Rossin.
\newblock A variant of the tandem duplication-random loss model of genome
  rearrangement.
\newblock {\em Theoretical Computer Science}, 410(8-10):847--858, 2009.

\bibitem{Kingman}
J.~F.~C. Kingman.
\newblock {\em Regenerative Phenomena}.
\newblock John {W}iley and {S}ons, London, 1972.

\bibitem{kitaev}
S.~Kitaev.
\newblock {\em Patterns in Permutations and Words}.
\newblock Springer, Berlin, 2011.

\bibitem{knuth}
D.~E. Knuth.
\newblock {\em The Art of Computer Programming, Volume 3}.
\newblock Addison-Wesley, Reading MA, 1973.

\bibitem{Kratt1}
C.~Krattenthaler.
\newblock Permutations with restricted patterns and {D}yck paths.
\newblock {\em Adv. Appl. Math.}, 27:510--530, 2001.

\bibitem{lawler}
G.~Lawler and V.~Limic.
\newblock {\em Random Walk: A Modern Introduction}.
\newblock Cambridge University Press, Cambridge, 2010.

\bibitem{madras2}
H.~Liu and N.~Madras.
\newblock Random pattern-avoiding permutations.
\newblock In M.E.~Lladser et~al., editor, {\em Algorithmic Probability and
  Combinatorics, Contemporary Mathematics, Vol.\ 520}, pages 173--194. Amer.\
  Math.\ Soc., 2010.

\bibitem{pak}
S.~Miner and I.~Pak.
\newblock The shape of random pattern avoiding permutations.
\newblock {\em Adv.\ Appl.\ Math.}, 55:86--130, 2014.

\bibitem{simion}
R.~Simion and F.W. Schmidt.
\newblock Restricted permutations.
\newblock {\em Europ. J. Combin.}, 6:383--406, 1985.

\bibitem{solomon}
N.~Solomon and S.~Solomon.
\newblock A natural extension of {C}atalan numbers.
\newblock {\em Journal of Integer Sequences}, 11:Article 08.3.5, 2008.

\end{thebibliography}
\end{document}